\documentclass[12pt,letterpaper,reqno]{amsart}

\usepackage{times}
\usepackage[T1]{fontenc}
\usepackage{mathrsfs}
\usepackage{latexsym}
\usepackage[dvips]{graphics}
\usepackage{epsfig}

\usepackage{amsmath,amsfonts,amsthm,amssymb,amscd}
\input amssym.def
\input amssym.tex

\newcommand{\supp}{{\rm supp}}
\renewcommand{\i}{{\mathrm{i}}} 

\newcommand{\hphipr}{\hphi\left(\frac{\log p}{\log R}\right)}

\newcommand{\hkn}{H_k^\ast(N)}

\newcommand{\un}{\text{U}}
\newcommand{\sy}{\text{USp}}

\newcommand{\soe}{\text{SO(even)}}
\newcommand{\soo}{\text{SO(odd)}}
\newcommand{\so}{\text{O}}

\newcommand{\hphi}{\widehat{\phi}}  

\newcommand\be{\begin{equation}}
\newcommand\ee{\end{equation}}
\newcommand\bea{\begin{eqnarray}}
\newcommand\eea{\end{eqnarray}}
\newcommand\bi{\begin{itemize}}
\newcommand\ei{\end{itemize}}
\newcommand\ben{\begin{enumerate}}
\newcommand\een{\end{enumerate}}
\newcommand\bc{\begin{center}}
\newcommand\ec{\end{center}}
\newcommand\ba{\begin{array}}
\newcommand\ea{\end{array}}

\def\notdiv{\ \mathbin{\mkern-8mu|\!\!\!\smallsetminus}}

\newcommand{\R}{\ensuremath{\mathbb{R}}}

\newcommand{\Z}{\ensuremath{\mathbb{Z}}}
\newcommand{\Q}{\mathbb{Q}}

\newcommand{\F}{\mathcal{F}}

\newcommand{\foh}{\frac{1}{2}}  

\newcommand{\es}[2]{\langle {#1 \atop #2}\rangle}
\newcommand{\zsum}[1]{ \sum_{#1 = 0}^{p-1} }

\newcommand{\js}[1]{{\underline{#1}\choose p}}

\newtheorem{thm}{Theorem}[section]

\newtheorem{lem}[thm]{Lemma}

\theoremstyle{definition}
\newtheorem{rek}[thm]{Remark}

\newcommand{\ncr}[2]{{#1 \choose #2}}
\newcommand{\twocase}[5]{#1 \begin{cases} #2 & \text{{\rm #3}}\\ #4
&\text{{\rm #5}} \end{cases}   }
\newcommand{\threecase}[7]{#1 \begin{cases} #2 & \text{{\rm #3}}\\ #4
&\text{{\rm #5}}\\ #6 &\text{{\rm #7}} \end{cases}   }

\newcommand{\gep}{\epsilon}
\newcommand{\gl}{\lambda}

\newcommand{\g}{\gamma}     

\newcommand{\fwf}{\frac1{W_R(\F)}}

\newcommand{\glfp}{\lambda_f(p)}

\newcommand{\gafp}{\alpha_f(p)}

\newcommand{\gbfp}{\beta_f(p)}
\newcommand{\glf}{\lambda_f}

\newcommand{\lix}[1]{{\rm Li}_{#1}(x)}

\begin{document}

\title[Lower order terms in $1$-level densities]{Lower
order terms in the $1$-level density for families of holomorphic
cuspidal newforms}

\author{Steven J. Miller}
\email{Steven.J.Miller@williams.edu} \address{Department of Mathematics and Statistics,
Williams College, Williamstown, MA 01267} 

\date{\today}

\thanks{The author would like to thank Walter Becker, Colin Deimer, Steven Finch, Dorian
Goldfeld, Filip Paun and Matt Young for useful discussions. Several of the formulas for
expressions in this paper were first guessed by using Sloane's
On-Line Encyclopedia of Integer Sequences \cite{Sl}. The numerical
computations were run on the Princeton Math Server, and it is a
pleasure to thank the staff there for their help. The author was
partly supported by NSF grant DMS0600848.}

\begin{abstract} The Katz-Sarnak density conjecture states that, in the limit
as the analytic conductors tend to infinity, the behavior of normalized zeros
near the central point of families of $L$-functions agree with the
$N\to\infty$ scaling limits of eigenvalues near $1$ of subgroups of
$U(N)$. Evidence for this has been found for many families by
studying the $n$-level densities; for suitably restricted test
functions the main terms agree with random matrix theory. In
particular, all one-parameter families of elliptic curves with rank
$r$ over $\Q(T)$ and the same distribution of signs of functional
equations have the same universal limiting behavior for their main term. We break this
universality and find family dependent lower order correction terms
in many cases; these lower order terms have applications ranging
from excess rank to modeling the behavior of zeros near the central
point, and depend on the arithmetic of the family. We derive an
alternate form of the explicit formula for ${\rm GL}(2)$
$L$-functions which simplifies comparisons, replacing sums over
powers of Satake parameters by sums of the moments of the Fourier
coefficients $\glfp$. Our formula highlights the differences that we
expect to exist from families whose Fourier coefficients obey
different laws (for example, we expect Sato-Tate to hold only for
non-CM families of elliptic curves). Further, by the work of Rosen
and Silverman we expect lower order biases to the Fourier
coefficients in one-parameter families of elliptic curves with rank over $\Q(T)$;
these biases can be seen in our expansions. We analyze several
families of elliptic curves and see different lower order
corrections, depending on whether or not the family has complex
multiplication, a forced torsion point, or non-zero rank over
$\Q(T)$.
\end{abstract}

\maketitle




\section{Introduction}
\setcounter{equation}{0}

Assuming the Generalized Riemann Hypothesis (GRH), the non-trivial zeros of any $L$-function have real
part equal to $1/2$. Initial investigations studied spacing
statistics of zeros far from the central point, where numerical
and theoretical results \cite{Hej,Mon,Od1,Od2,RS} showed excellent
agreement with eigenvalues from the Gaussian Unitary Ensemble (GUE). Further agreement
was found in studying moments of $L$-functions
\cite{CF,CFKRS,KeSn1,KeSn2,KeSn3} as well as low-lying zeros (zeros
near the critical point).

In this paper we concentrate on low-lying zeros of $L(s,f)$, where $f \in H^\star_k(N)$, the set of all holomorphic cuspidal newforms of weight $k$ and level $N$. Before stating our results, we briefly review some notation and standard facts. Each $f\in H^\star_k(N)$ has a Fourier expansion \begin{equation}\label{eq:defLsf1}
f(z)\ =\ \sum_{n=1}^\infty a_f(n) e(nz).
\end{equation}
Let $\lambda_f(n) =  a_f(n) n^{-(k-1)/2}$. These coefficients
satisfy multiplicative relations, and $|\lambda_f(p)| \le 2$. The
$L$-function associated to $f$ is
\begin{equation}
L(s,f)\ =\ \sum_{n=1}^\infty \frac{\lambda_f(n)}{n^{s}} \ = \
\prod_p \left(1 - \frac{\lambda_f(p)}{p^s} +
\frac{\chi_0(p)}{p^{2s}}\right)^{-1},
\end{equation} where $\chi_0$ is the principal character with
modulus $N$. We write \be \lambda_f(p) \ = \ \alpha_f(p) +
\beta_f(p). \ee For $p\ \notdiv N$, $\alpha_f(p)\beta_f(p) = 1$ and
$|\alpha_f(p)| = 1$. If $p|N$ we take $\alpha_f(p) = \lambda_f(p)$
and $\beta_f(p) = 0$. Letting \bea L_\infty(s,f) & \ = \ &
\left(\frac{2^k}{8\pi}\right)^{1/2}\
\left(\frac{\sqrt{N}}{\pi}\right)^s\
\Gamma\left(\frac{s}2+\frac{k-1}4\right)\
\Gamma\left(\frac{s}2+\frac{k+1}4\right) \eea denote the local
factor at infinity, the completed $L$-function is
\begin{equation}\label{eq:defLsf2}
\Lambda(s,f) \ =\ L_\infty(s) L(s,f) \ =\ \epsilon_f \Lambda(1-s,f),
\ \ \ \epsilon_f = \pm 1.\end{equation} Therefore $H^\star_k(N)$
splits into two disjoint subsets, $H^+_k(N) = \{ f\in H^\star_k(N):
\epsilon_f = +1\}$ and $H^-_k(N) = \{ f\in H^\star_k(N): \epsilon_f
= -1\}$. Each $L$-function has a set of non-trivial zeros
$\rho_{f,j} = \tfrac12 + \i\g_{f,\ell}$. The Generalized Riemann
Hypothesis asserts that all $\g_{f,\ell} \in \R$.

In studying the behavior of low-lying zeros, the arithmetic and analytic conductors determine the appropriate scale. For $f \in H^\star_k(N)$, the arithmetic conductor $N_f$ is the integer $N$ from the functional equation, and the analytic conductor $Q_f$ is $\frac{N}{\pi^2} \frac{(k+1)(k+3)}{64}$. The number of zeros within $C$ units of the central point (where $C$ is any large, absolute constant) is of the order $\log Q_f$. For us $k$ will always be fixed, so $N_f$ and $Q_f$ will  differ by a constant. Thus $\log Q_f \sim \log N_f$, and in the limit as the level $N$ tends to infinity, we may use either the analytic or arithmetic conductor to normalize the zeros near the central point. See \cite{Ha,ILS} for more details.

We rescale the zeros and study $\gamma_{f,\ell}  \frac{\log Q_f}{2\pi}$. We let $\mathcal{F} = \cup \mathcal{F}_N$ be a family of $L$-functions ordered by conductor (our first example will be $\mathcal{F}_N = H^\ast_k(N)$; later we shall consider one-parameter families of elliptic curves). The
$n$-level density for the family is
\begin{eqnarray}
D_{n,\mathcal{F}}(\phi)\ :=\ \lim_{N\to\infty} \frac{1}{|\mathcal{F}_N|} \sum_{f\in
\mathcal{F}_N} \sum_{\substack{\ell_1,\dots, \ell_n \\ \ell_i \neq \pm
\ell_k}} \phi_1\left(\gamma_{f,\ell_1}\frac{\log
Q_f}{2\pi}\right)\cdots \phi_n\left(\gamma_{f,\ell_n}\frac{\log
Q_f}{2\pi}\right),
\end{eqnarray} where the $\phi_i$ are even Schwartz test functions
whose Fourier transforms have compact support and $\foh +
\i\gamma_{f,\ell}$ runs through the non-trivial zeros of $L(s,f)$. As the $\phi_i$'s are even
Schwartz functions, most of the contribution to
$D_{n,\mathcal{F}}(\phi)$ arises from the zeros  near the central
point; thus this statistic is well-suited to investigating the
low-lying zeros. For some families, it is more convenient to incorporate weights (for example,
the harmonic weights facilitate applying the Petersson formula to
families of cuspidal newforms).

Katz and Sarnak \cite{KaSa1,KaSa2} conjectured that, in the limit as
the analytic conductors tend to infinity, the behavior of the normalized
zeros near the central point of a family $\mathcal{F}$ of $L$-functions agrees with the $N\to\infty$ scaling
limit of the normalized eigenvalues near $1$ of a subgroup of
$U(N)$: \bea
D_{n,\mathcal{F}}(\phi)& \ =\ &  \int \cdots \int
\phi_1(x_1)\cdots \phi_n(x_n)
W_{n,G(\mathcal{F})}(x_1,\dots,x_n)dx_1\cdots dx_n, \eea where
$G(\mathcal{F})$ is the scaling limit of one of the following classical compact groups: $N\times N$ unitary,
symplectic or orthogonal matrices.\footnote{For test functions
$\hphi$ supported in $(-1,1)$, the one-level densities are
\be\label{eq:1ldwithtestfn}
\begin{array}{lcl}
\int \phi(u) {W}_{1,\soe}(u)du & = & \hphi(u) + \foh \phi(0)\\
\int \phi(u){W}_{1,\soo}(u)du & = & \hphi(u) + \foh \phi(0)
\\ \int \phi(u){W}_{1,\so}(u)du & = & \hphi(u) +
\foh\phi(0) \\ \int \phi(u){W}_{1,\sy}(u)du & = & \hphi(u) - \foh
\phi(0) \\ \int \phi(u){W}_{1,\un}(u)du & = & \hphi(u).
\end{array}\ee }
Evidence towards this conjecture is provided by analyzing the $n$-level densities of
many families, such as all Dirichlet characters, quadratic Dirichlet
characters, $L(s,\psi)$ with $\psi$ a character of the ideal class
group of the imaginary quadratic field $\mathbb{Q}(\sqrt{-D})$,
families of elliptic curves, weight $k$ level $N$ cuspidal newforms,
symmetric powers of ${\rm GL}(2)$ $L$-functions, and certain
families of ${\rm GL}(4)$ and ${\rm GL}(6)$ $L$-functions; see
\cite{DM1,FI,Gu,HR,HM,ILS,KaSa2,Mil2,OS,RR1,Ro,Rub,Yo2}.

Different classical compact groups exhibit a different local
behavior of eigenvalues near $1$, thus breaking the global GUE
symmetry. This correspondence allows us, at least conjecturally, to
assign a definite ``symmetry type'' to each family of primitive
$L$-functions.\footnote{For families of zeta or $L$-functions of curves or varieties over finite fields, the corresponding classical compact group can be determined by the monodromy (or symmetry group) of the family and its scaling limit. No such identification is known for number fields, though function field analogues often suggest what the symmetry type should be. See also \cite{DM2} for results about
the symmetry group of the convolution of families, as well as
determining the symmetry group of a family by analyzing the second
moment of the Satake parameters.}

Now that the main terms have been shown to agree with random matrix
theory predictions (at least for suitably restricted test
functions), it is natural to study the lower order terms.\footnote{Recently Conrey, Farmer and Zirnbauer
\cite{CFZ1,CFZ2} conjectured formulas for the averages over a family
of ratios of products of shifted $L$-functions. Their $L$-functions
Ratios Conjecture predicts both the main and lower order terms for
many problems, ranging from $n$-level correlations and densities to
mollifiers and moments to vanishing at the central point (see
\cite{CS}). In \cite{Mil6, Mil7} we verified the Ratios Conjecture's predictions (up to error terms of size $O(X^{-1/2+\gep})$!) for the
$1$-level density of the family of quadratic Dirichlet characters and certain families of cuspidal newforms for test functions of suitably small support. Khiem is currently calculating the predictions of the
Ratios Conjecture for certain families of elliptic curves.} In this
paper we see how various arithmetical properties of families of
elliptic curves (complex multiplication, torsion groups, and rank)
affect the lower order terms. \footnote{While the main terms for one-parameter families of
elliptic curves of rank $r$ over $\Q(T)$ and given distribution of
signs of functional equations all agree with the scaling limit of
the same orthogonal group, in \cite{Mil1} potential lower order
corrections were observed (see \cite{FI,RR2,Yo1} for additional
examples, and \cite{Mil3} for applications of lower order terms to
bounding the average order of vanishing at the central point in a
family). The problem is that these terms are of size $1/\log R$,
while trivially estimating terms in the explicit formula lead to
errors of size $\log\log R/\log R$; here $\log R$ is the average log-conductor of the family. These lower order terms are
useful in refining the models of zeros near the central point for
small conductors. This is similar to modeling high zeros of
$\zeta(s)$ at height $T$ with matrices of size $N=\log (T/2\pi)$
(and not the $N \to \infty$ scaling limits) \cite{KeSn1,KeSn2}; in
fact, even better agreement is obtained by a further adjustment of
$N$ arising from an analysis of the lower order terms (see
\cite{BBLM,DHKMS}).} For families of elliptic curves these lower order terms have appeared in excess rank investigations \cite{Mil3}, and in a later paper \cite{DHKMS} they will play a role in explaining the repulsion observed in \cite{Mil4} of the first normalized zero above the central point in one-parameter families of elliptic curves.

We derive an alternate version of the explicit formula for a
family $\mathcal{F}$ of ${\rm GL}(2)$ $L$-functions of weight $k$
which is more tractable for such investigations, which immediately yields a useful expansion for the $1$-level density for a
family $\F$ of ${\rm GL}(2)$ cuspidal newforms. \emph{We should really write $\mathcal{F}_N$ and $R_N$ below to emphasize that our calculations are being done for a fixed $N$, and then take the limit as $N\to\infty$. As there is no danger of confusion, we suppress the $N$ in the $\mathcal{F}_N$ and $R_N$.}

Let $N_f$ be the
level of $f\in\mathcal{F}$ and let $\phi$ be an even Schwartz
function such that $\hphi$ has compact support, say $\supp(\hphi)
\subset (-\sigma, \sigma)$. We weight each $f\in\mathcal{F}$ by
non-negative weights $w_R(f)$, where $\log R$ is the
weighted average of the logarithms of the levels, and we rescale the
zeros near the central point by $(\log R)/2\pi$ (in all our families of interest, $\log R \sim \log N$). Set $W_R(f) =
\sum_{f\in\F} w_R(f)$. The $1$-level density for the family $\F$
with weights $w_R(f)$ and test function $\phi$ is
\bea\label{eq:oneleveldensity} D_{1,\F}(\phi) & = & \frac1{W_R(\F)}
\sum_{f\in\F} w_R(f) \sum_j \phi\left(\gamma_{f,\ell}\frac{\log
R}{2\pi}\right) \nonumber\\ &=& \frac{\sum_{f\in \F} w_R(f)
(A(k)+\log
N_f)}{W_R(\F)\log R}\ \hphi(0)  \nonumber\\
& & -\ 2 \sum_p \sum_{m=1}^\infty \frac1{W_R(\F)}
\sum_{f\in\F}w_R(f)\frac{\alpha_f(p)^m + \beta_f(p)^m}{p^{m/2}}
\frac{\log p}{\log R}\ \hphi\left(m\frac{\log p}{\log R}\right)
 +\ O_k\left(\frac1{\log^2 R}\right) \nonumber\\ &=&
\frac{\sum_{f\in \F} w_R(f) (A(k)+\log N_f)}{W_R(\F)\log R}\
\hphi(0) + S(\F) + O_k\left(\frac1{\log^2 R}\right), \eea with
$\psi(z) = \Gamma'(z)/\Gamma(z)$, $A(k) = \psi(k/4) + \psi((k+2)/4)
- 2\log \pi$, and \bea\label{eq:defSF} S(\F) \ = \ -\ 2 \sum_p
\sum_{m=1}^\infty \frac1{W_R(\F)}
\sum_{f\in\F}w_R(f)\frac{\alpha_f(p)^m + \beta_f(p)^m}{p^{m/2}}
\frac{\log p}{\log R}\ \hphi\left(m\frac{\log p}{\log R}\right).
\eea The above is a straightforward consequence of the explicit
formula, and depends crucially on having an Euler product for our
$L$-functions; see \cite{ILS} for a proof. As $\phi$ is a Schwartz
function, most of the contribution is due to the zeros near the
central point. The error of size $1/\log^2 R$ arises from
simplifying some of the expressions involving the analytic
conductors, and could be improved to be of size $1/\log^3 R$ at the
cost of additional analysis (see \cite{Yo1} for details); as we are
concerned with lower order corrections due to arithmetic differences
between the families, the above suffices for our purposes.

The difficult (and interesting) piece in the $1$-level density is
$S(\F)$. Our main result is an alternate version of the explicit
formula for this piece. We first set the notation. For each
$f\in\F$, let \be\label{eq:defnSp} S(p) \ = \ \{f \in \F: p\ \notdiv
N_f\}. \ee Thus for $f\notin S(p)$, $\gafp^m+\gbfp^m=\glfp^m$. Let
\bea\label{eq:defArfpArfprimep} A_{r,\F}(p)  \ = \ \fwf\sum_{f \in
\F \atop f \in S(p)} w_R(f)\glfp^r,\ \ \ \ A_{r,\F}'(p) \ =\
\fwf\sum_{f \in \F \atop f \notin S(p)} w_R(f)\glfp^r; \eea \emph{we
use the convention that $0^0 = 1$; thus $A_{0,\F}(p)$ equals the
cardinality of $S(p)$.}

\begin{thm}[Expansion for $S(\F)$ in terms of moments of $\glfp$]\label{thm:expSF} Let $\log R$ be the average log-conductor of a finite family of $L$-functions $\F$, and let $S(\F)$ be as in \eqref{eq:defSF}.
We have \bea S(\F) & \ = \ & -\ 2 \sum_p \sum_{m=1}^\infty
\frac{A_{m,\F}'(p)}{p^{m/2}} \frac{\log p}{\log R}\
\hphi\left(m\frac{\log p}{\log R}\right)\nonumber\\ & &
-2\hphi(0)\sum_p \frac{2A_{0,\F}(p)\log p} {p(p+1)\log R}\ +\
2\sum_p \frac{2A_{0,\F}(p)\log p}{p\log R}\ \hphi\left(
2\frac{\log p}{\log R}\right)\nonumber\\
& &-2\sum_p \frac{A_{1,\F}(p)}{p^{1/2}} \frac{\log p}{\log R}\
\hphipr + 2 \hphi(0) \frac{A_{1,\F}(p)
(3p+1)}{p^{1/2}(p+1)^2}\frac{\log p}{\log R} \nonumber\\ & &
-2\sum_p \frac{A_{2,\F}(p)\log p}{p\log R}\ \hphi\left(2\frac{\log
p}{\log R}\right) + 2\hphi(0)\sum_p \frac{A_{2,\F}(p)(4p^2+3p+1)\log
p}{p(p+1)^3\log R} \nonumber\\ & & - 2\hphi(0)\sum_{p}
\sum_{r=3}^\infty \frac{A_{r,\F}(p)p^{r/2}(p-1)\log
p}{(p+1)^{r+1}\log R} \ + \ O\left(\frac1{\log^3 R}\right)
\nonumber\\ & =&  S_{A'}(\mathcal{F}) + S_0(\F)+S_1(\F)+S_2(\F) +
S_A(\F)+O\left(\frac1{\log^3 R}\right). \eea If we let \be
\widetilde A_\F(p) \ = \ \fwf \sum_{f \in S(p)} w_R(f)
\frac{\glf(p)^3}{p+1 - \glfp\sqrt{p}}, \ee then by the geometric
series formula we may replace $S_A(\F)$ with $S_{\tilde A}(\F)$,
where \be S_{\tilde A}(\F) \ = \ - 2\hphi(0)\sum_{p}
\frac{\widetilde A_\F(p) p^{3/2}(p-1)\log p}{(p+1)^3 \log R}. \ee
\end{thm}

\begin{rek}\label{rek:troubleoneparamecs}
For a general one-parameter family of elliptic curves, we are unable
to obtain exact, closed formulas for the $r$\textsuperscript{th}
moment terms $A_{r,\F}(p)$; for sufficiently nice families we can
find exact formulas for $r \le 2$ (see \cite{ALM,Mil3} for some
examples, with applications towards constructing families with
moderate rank over $\Q(T)$ and the excess rank question). Thus we
are forced to numerically approximate the $A_{r,\F}(p)$ terms when
$r\ge 3$.\footnote{This greatly hinders comparison with the $L$-Functions Ratios Conjecture, which gives useful interpretations for the lower order terms. In \cite{CS} the lower order terms are computed for a symplectic family of quadratic Dirichlet $L$-functions. The
(conjectured) expansions there show a remarkable relation between
the lower order terms and the zeros of the Riemann zeta function;
for test functions with suitably restricted support, the number
theory calculations are tractable and in \cite{Mil6} are shown to
agree with the Ratios Conjecture.}
\end{rek}

We prove Theorem \ref{thm:expSF} by using the geometric series
formula for $\sum_{m \ge 3} (\gafp/\sqrt{p})^m$ (and similarly for
the sum involving $\gbfp^m$) and properties of the Satake
parameters. We find terms like \be \frac1{p^{3/2}} \frac{\glfp^3 -
3\glfp}{p+1 - \glfp\sqrt{p}} \ - \ \frac1{p^2} \frac{\glfp^2
-2}{p+1-\glfp\sqrt{p}}. \ee While the above formula leads to
tractable expressions for computations, the disadvantage is that the
zeroth, first and second moments of $\glfp$ are now weighted by
$1/(p+1-\glfp\sqrt{p})$. For many families (especially those of
elliptic curves) we can calculate the zeroth, first and second
moments exactly up to errors of size $1/N^\gep$; this is not the
case if we introduce these weights in the denominator. We therefore
apply the geometric series formula again to expand
$1/(p+1-\glfp\sqrt{p})$ and collect terms.

An alternate proof involves replacing each $\gafp^m+\gbfp^m$ for
$p\in S(p)$ with a polynomial $\sum_{r=0}^m c_{m,r} \glfp^m$, and
then interchanging the order of summation (which requires some work,
as the resulting sum is only conditionally convergent). The sum over
$r$ collapses to a linear combination of polylogarithm
functions, and the proof is
completed by deriving an identity expressing these sums as a simple
rational function.\footnote{The polylogarithm function is $\lix{s}=
\sum_{k=1}^\infty k^{-s} x^k$. If $s$ is a negative integer, say
$s=-r$, then the polylogarithm function converges for $|x|<1$ and
equals $\sum_{j=0}^r \es{r}{j} x^{r-j}\Big/(1-x)^{r+1}$, where the
$\es{r}{j}$ are the Eulerian numbers (the number of permutations of
$\{1,\dots,r\}$ with $j$ permutation ascents). In \cite{Mil5} we show that if $a_{\ell,i}$ is the coefficient of
$k^{i}$ in $\prod_{j=0}^{\ell-1} (k^2-j^2)$, and $b_{\ell,i}$ is
the coefficient of $k^{i}$ in $(2k+1)\prod_{j=0}^{\ell-1}
(k-j)(k+1+j)$, then for $|x| < 1$ and $\ell \ge 1$ we have \bea
a_{\ell,2\ell} \lix{-2\ell} + \cdots + a_{\ell,0}\lix{0} &\ =
\ &  \frac{(2\ell)!}{2}\ \frac{x^\ell(1+x)}{(1-x)^{2\ell+1}} \nonumber\\
b_{\ell,2\ell+1} \lix{-2\ell-1} + \cdots + b_{\ell,0}\lix{0}& \ = \
& (2\ell+1)!\ \frac{x^\ell(1+x)}{(1-x)^{2\ell+2}}. \eea
Another application of this identity is to deduce relations among the Eulerian numbers.}

\begin{rek} An advantage of the explicit formula in Theorem
\ref{thm:expSF} is that the answer is expressed as a weighted sum of
moments of the Fourier coefficients. Often much is known (either
theoretically or conjecturally) for the distribution of the Fourier
coefficients, and this formula facilitates comparisons with
conjectures. In fact, often the $r$-sum can be collapsed by using
the generating function for the moments of $\glfp$. Moreover, there
are many situations where the Fourier coefficients are easier to
compute than the Satake parameters; for elliptic curves we find the
Fourier coefficients by evaluating sums of Legendre symbols, and
then pass to the Satake parameters by solving $a_E(p) = 2 \sqrt{p}
\cos \theta_E(p)$. Thus it is convenient to have the formulas in
terms of the Fourier coefficients. As $\widetilde A_\F(p) =
O(1/p)$), these sums converge at a reasonable rate, and we can
evaluate the lower order terms of size $1/\log R$ to any specified
accuracy by simply calculating moments and modified moments of the
Fourier coefficients at the primes.
\end{rek}

We now summarize the lower order terms for several different
families of ${\rm GL}(2)$ $L$-functions; many other families can be
computed through these techniques. The first example is analyzed in
\S\ref{sec:cuspnewS}, the others in
\S\ref{sec:oneparamfamellcurvesS}. Below we merely state the final
answer of the size of the $1/\log R$ term to a few digits accuracy; see the relevant sections for expressions of these constants
in terms of prime sums with weights depending on the family. For
sufficiently small support, the main term in the $1$-level density
of each family has previously been shown to agree with the three
orthogonal groups (we can determine which by calculating the
$2$-level density and splitting by sign); however, the lower order
terms are different for each family, showing how the arithmetic of
the family enters as corrections to the main term. For most of our
applications we have weight $2$ cuspidal newforms, and thus the
conductor-dependent terms in the lower order terms are the same for
all families. Therefore below we shall only describe the
family-dependent corrections.

\bi \item \textbf{All holomorphic cusp forms (Theorem \ref{thm:cuspnewformslot}):} Let
$\mathcal{F}_{k,N}$ be either the family of even weight $k$ and
prime level $N$ cuspidal newforms, or just the forms with even (or
odd) functional equation. Up to $O(\log^{-3} R)$, for test functions
$\phi$ with $\supp(\hphi) \subset (-4/3, 4/3)$, as $N\to\infty$ the
(non-conductor) lower order term is  \be -1.33258\cdot
2\hphi(0)/\log R. \ee Note the lower order corrections are
independent of the distribution of the signs of the functional
equations. \\

\item \textbf{CM example, with or without
forced torsion (Theorem \ref{thm:ellcurvemainb1thru6kappa}):} Consider the one-parameter families $y^2 = x^3 +
B(6T+1)^\kappa$ over $\Q(T)$, with $B \in \{1,2,3,6\}$ and $\kappa
\in \{1,2\}$; these families have complex multiplication, and
thus the distribution of their Fourier coefficients does not follow
Sato-Tate. We sieve so that $(6T+1)$ is $(6/\kappa)$-power free. If
$\kappa = 1$ then all values of $B$ have the same behavior, which is
very close to what we would get if the average of the Fourier
coefficients immediately converged to the correct limiting
behavior.\footnote{In practice, it is only as $p\to\infty$ that the
average moments converge to the complex multiplication distribution;
for finite $p$ the lower order terms to these moments mean that the
answer for families of elliptic curves with complex multiplication
is not the same as what we would obtain by replacing these averages
with the moments of the complex multiplication distribution.} If
$\kappa=2$ the four values of $B$ have different lower order
corrections; in particular, if $B=1$ then there is a forced torsion
point of order three, $(0,6T+1)$. Up to errors of size $O(\log^{-3}
R)$, the (non-conductor) lower order terms are approximately \bea B
= 1,\kappa = 1: &\ \ & -2.124 \cdot 2\hphi(0)/\log R, \nonumber\\ B
= 1,\kappa = 2: &\ \ & -2.201 \cdot 2\hphi(0)/\log R, \nonumber\\
B = 2,\kappa = 2: &\ \ & -2.347 \cdot 2\hphi(0)/\log R \nonumber\\
B = 3,\kappa = 2: &\ \ & -1.921 \cdot 2\hphi(0)/\log R \nonumber\\
B = 6,\kappa = 2: &\ \ & -2.042 \cdot 2\hphi(0)/\log R. \eea \

\item \textbf{CM example, with or without rank (see \S\ref{sec:family00036t65x0}):} Consider the one-parameter
families $y^2 = x^3 - B(36T+6)(36T+5)x$ over $\Q(T)$, with $B\in
\{1,2\}$. If $B=1$ the family has rank 1, while if $B=2$ the family
has rank 0; in both cases the family has complex multiplication. We
sieve so that $(36T+6)(36T+5)$ is cube-free. The most important
difference between these two families is the contribution from the
$S_{\widetilde{\mathcal{A}}}(\mathcal{F})$ terms, where the $B=1$
family is approximately $-.11 \cdot 2\hphi(0)/\log R$, while the
$B=2$ family is approximately $.63\cdot 2\hphi(0)/\log R$. This
large difference is due to biases of size $-r$ in the Fourier
coefficients $a_t(p)$ in a one-parameter family of rank $r$ over $\Q(T)$. Thus, while the main
term of the average moments of the $p$\textsuperscript{th} Fourier
coefficients are given by the complex multiplication analogue of
Sato-Tate in the limit, for each $p$ there are lower order
correction terms which depend on the rank. This is in line with
other results. Rosen and Silverman \cite{RoSi} prove $\sum_{t\bmod
p} a_t(p)$ is related to the negative of the rank of the family over
$\Q(T)$; see Theorem \ref{thmsr} for
an exact statement.\\

\item \textbf{Non-CM Example (see Theorem \ref{thm:sdfafadadafafa}):} Consider the one-parameter family $y^2 = x^3
- 3x+12T$ over $\Q(T)$. Up to $O(\log^{-3} R)$, the (non-conductor)
lower order correction is approximately \be -2.703 \cdot
2\hphi(0)/\log R,\ee which is very different than the family of
weight
$2$ cuspidal newforms of prime level $N$.\\

\ei

\begin{rek} While the main terms of the $1$-level density in these families depend only weakly on the family,\footnote{All that matters are the first two
moments of the Fourier coefficients. All families have the same main
term in the second moments; the main term in the first moment is
just the rank of the family. See \cite{Mil2} for details for
one-parameter families of elliptic curves} we see
that the lower order correction terms depend on finer arithmetical
properties of the family. In particular, we see differences
depending on whether or not there is complex multiplication, a
forced torsion point, or rank. Further, the lower order correction
terms are more negative for families of elliptic curves with forced
additive reduction at $2$ and $3$ than for all cuspidal newforms of
prime level $N\to\infty$. This is similar to Young's results
\cite{Yo1}, where he considered two-parameter families and noticed
that the number of primes dividing the conductor is negatively
correlated to the number of low-lying zeros. A better comparison
would perhaps be to square-free $N$ with the number of factors
tending to infinity, arguing as in \cite{ILS} to handle the
necessary sieving.
\end{rek}

\begin{rek} The proof of the Central Limit Theorem provides a useful
analogy for our results. If $X_1, \dots, X_N$ are `nice'
independent, identically distributed random variables with mean
$\mu$ and variance $\sigma^2$, then as $N\to\infty$ we have
$(X_1+\cdots+X_N - N\mu)/\sigma\sqrt{N}$ converges to the standard
normal. The universality is that, properly normalized, the main term
is independent of the initial distribution; however, the rate of
convergence to the standard normal depends on the higher moments of
the distribution. We observe a similar phenomenon with the $1$-level
density. We see universal answers (agreeing with random matrix
theory) as the conductors tend to infinity in the main terms; however, the rate of
convergence (the lower order terms) depends on the higher moments of
the Fourier coefficients.
\end{rek}

The paper is organized as follows. In \S\ref{sec:standardexpform} we
review the standard explicit formula and then prove our alternate
version (replacing averages of Satake parameters with averages of
the Fourier coefficients). We analyze all cuspidal newforms in
\S\ref{sec:cuspnewS}. After some preliminary expansions for elliptic
curve families in \S\ref{sec:prelimECfam}, we analyze several
one-parameter families in \S\ref{sec:oneparamfamellcurvesS}.


\section{Explicit Formulas}\label{sec:standardexpform}
\setcounter{equation}{0}

\subsection{Standard Explicit Formula}

Let $\phi$ be an even Schwartz test function whose Fourier transform
has compact support, say $\supp(\hphi) \subset (-\sigma,\sigma)$. Let
$f$ be a weight $k$ cuspidal newform of level $N$; see
\eqref{eq:defLsf1} through \eqref{eq:defLsf2} for a review of
notation. The explicit formula relates sums of $\phi$ over the zeros
of $\Lambda(s,f)$ to sums of $\hphi$ and the Fourier coefficients
over prime powers. We have (see for example Equations (4.11)--(4.13)
of \cite{ILS}) that \bea\label{eq:expformonef1} \sum_\gamma
\phi\left(\gamma\frac{\log R}{2\pi}\right) & \ = \ &
\frac{A_{k,N}(\phi)}{\log R} - 2 \sum_p \sum_{m=1}^\infty
\frac{\alpha_f(p)^m + \beta_f(p)^m}{p^{m/2}} \frac{\log p}{\log R}
\hphi\left(m\frac{\log p}{\log R}\right), \nonumber\\ \eea where
\bea A_{k,N}(\phi) & \ = \ & 2\hphi(0)
\log\left(\frac{\sqrt{N}}{\pi}\right) + \sum_{j=1}^2
A_{k,N;j}(\phi), \nonumber\\ A_{k,N;j}(\phi) & = &
\int_{-\infty}^\infty \psi\left(\alpha_j+\frac14+\frac{2\pi\i
x}{\log R}\right) \phi(x)dx, \nonumber\\ \eea with $\psi(z) =
\Gamma'(z)/\Gamma(z)$, $\alpha_1 = \frac{k-1}4$ and $\alpha_2 =
\frac{k+1}4$.

In this paper we concentrate on the first order correction terms to
the $1$-level density. Thus we are isolating terms of size $1/\log
R$, and ignoring terms that are $O(1/\log^2R)$. While a more careful
analysis (as in \cite{Yo1}) would allow us to analyze these
conductor terms up to an error of size $O(\log^{-3} R)$, these
additional terms are independent of the family and thus not as
interesting for our purposes. We use (8.363.3) of \cite{GR} (which
says $\psi(a+b\i) + \psi(a-b\i) = 2\psi(a) + O(b^2/a^2)$ for $a,b$
real and $a>0$), and find \be A_{k,N;j}(\phi) \ = \
\hphi(0)\psi\left(\alpha_j+\frac14\right) +
O\left(\frac1{(\alpha_j+1)^2\log^2 R}\right). \ee This implies that
\bea A_{k,N}(\phi) & \ = \ & \hphi(0)\log N
+\hphi(0)\left(\psi\left(\frac{k}{4}\right) +
\psi\left(\frac{k+2}4\right) - 2\log \pi \right) \nonumber\\ & & \ \
+\ O\left(\frac1{(\alpha_j+1)^2\log^2 R}\right).\eea As we shall
consider the case of $k$ fixed and $N\to\infty$, the above expansion
suffices for our purposes and we write \be A_{k,N}(\phi) \ = \
\hphi(0)\log N + \hphi(0) A(k) + O_k\left(\frac1{\log^2R}\right).\ee

We now average \eqref{eq:expformonef1} over all $f$ in our family
$\F$. We allow ourselves the flexibility to introduce slowly varying
non-negative weights $w_R(f)$, as well as allowing the levels of the
$f\in\F$ to vary. This yields the expansion for the $1$-level
density for the family, which is given by
\eqref{eq:oneleveldensity}.

We have freedom to choose the weights $w_R(f)$ and the scaling
parameter $R$. For families of elliptic curves we often take the
weights to be $1$ for $t \in [N,2N]$ such that the irreducible
polynomial factors of the discriminant are square or cube-free, and
zero otherwise (equivalently, so that the specialization $E_t$
yields a global minimal Weierstrass equation); $\log R$ is often the
average log-conductor (or a close approximation to it). For families
of cuspidal newforms of weight $k$ and square-free level $N$ tending
to infinity, we might take $w_R(f)$ to be the harmonic weights (to
simplify applying the Petersson formula) and $R$ around $k^2 N$
(i.e., approximately the analytic conductor).

The interesting piece in \eqref{eq:oneleveldensity} is \be S(\F) \ =
\ -\ 2 \sum_p \sum_{m=1}^\infty \frac1{W_R(\F)}
\sum_{f\in\F}w_R(f)\frac{\alpha_f(p)^m + \beta_f(p)^m}{p^{m/2}}
\frac{\log p}{\log R}\ \hphi\left(m\frac{\log p}{\log R}\right). \ee
We rewrite the expansion above in terms of the moments of the
Fourier coefficients $\glf(p)$. If $p|N_f$ then $\gafp^m + \gbfp^m =
\glfp^m$. Thus \bea S(\F) & \ = \ & -\ 2 \sum_p \sum_{m=1}^\infty
\frac1{W_R(\F)} \sum_{f\in\F \atop
p|N_f}w_R(f)\frac{\glfp^m}{p^{m/2}} \frac{\log p}{\log R}\
\hphi\left(m\frac{\log p}{\log R}\right)\nonumber\\ & & -\ 2 \sum_p
\sum_{m=1}^\infty \frac1{W_R(\F)} \sum_{f\in\F\atop p\ \notdiv
N_f}w_R(f)\frac{\alpha_f(p)^m + \beta_f(p)^m}{p^{m/2}} \frac{\log
p}{\log R}\ \hphi\left(m\frac{\log p}{\log R}\right).\nonumber\\
\eea

In the explicit formula we have terms such as $\hphi(m\log p/\log
R)$. As $\hphi$ is an even function, Taylor expanding gives
\be\label{eq:taylorexpandhphi} \hphi\left(m\frac{\log p}{\log
R}\right) \ = \ \hphi(0) + O\left(\left(m\frac{\log p}{\log
R}\right)^2\right). \ee As we are isolating lower order correction
terms of size $1/\log R$ in $S(\F)$, we ignore any term which
is $o(1/\log R)$. We therefore may replace $\hphi(m\log p/\log R)$
with $\hphi(\log p/\log R)$ at a cost of $O(1/\log^3 R)$ for all
$m\ge 3$,\footnote{As $\hphi$ has compact support, the only $m$ that contribute are $m \ll \log R$, and thus we do not need to worry about the $m$-dependence in this approximation because these terms are hit by a $p^{-m/2}$.} which yields \bea\label{eq:expSFearly1} S(\F) & = & -\ 2
\sum_p \sum_{m=1}^\infty \frac1{W_R(\F)} \sum_{f\in\F \atop
p|N_f}w_R(f)\frac{\glfp^m}{p^{m/2}} \frac{\log p}{\log R}\
\hphi\left(m\frac{\log p}{\log R}\right)\nonumber\\ & & -\ 2 \sum_p
\frac1{W_R(\F)} \sum_{f\in\F \atop p\ \notdiv
N_f}w_R(f)\frac{\glf(p)}{p^{1/2}} \frac{\log p}{\log R}\
\hphi\left(\frac{\log p}{\log R}\right)\nonumber\\ & & -\ 2 \sum_p
\frac1{W_R(\F)} \sum_{f\in\F\atop p\ \notdiv
N_f}w_R(f)\frac{\glf(p)^2-2}{p} \frac{\log p}{\log R}\
\hphi\left(2\frac{\log p}{\log R}\right)\nonumber\\ & & -\ 2 \sum_p
\sum_{m=3}^\infty \frac1{W_R(\F)} \sum_{f\in\F\atop p\ \notdiv
N_f}w_R(f)\frac{\alpha_f(p)^m + \beta_f(p)^m}{p^{m/2}} \frac{\log
p}{\log R}\ \hphipr\ + \ O\left(\frac1{\log^3 R}\right).\nonumber\\ \eea We have isolated the $m =1$ and $2$ terms from $p\
\notdiv N_f$ as these can contribute main terms (and not just lower
order terms). We used for $p\ \notdiv N_f$ that $\gafp +
\gbfp=\glfp$ and $\gafp^2 + \gbfp^2=\glf(p)^2-2$.

\subsection{The Alternate Explicit
Formula}\label{sec:proofaltexpform}

\begin{proof}[Proof of Theorem \ref{thm:expSF}]

We use the geometric series formula for the $m \ge 3$ terms in
\eqref{eq:expSFearly1}. We have \bea M_3(p)\ := \ \sum_{m=3}^\infty
\left[\left(\frac{\gafp}{\sqrt{p}}\right)^m +
\left(\frac{\gbfp}{\sqrt{p}}\right)^m\right] & \ = \ &
\frac{\gafp^3}{p(\sqrt{p}-\gafp)} +
\frac{\gbfp^3}{p(\sqrt{p}-\gbfp)} \nonumber\\ & = &
\frac{(\gafp^3+\gbfp^3)\sqrt{p}-(\gafp^2+\gbfp^2)}{p(p+1-\glfp\sqrt{p})}
\nonumber\\ & = & \frac{\glfp^3\sqrt{p} - \glfp^2 -
3\glfp\sqrt{p}+2}{p(p+1-\glfp\sqrt{p})},\nonumber\\\eea where we use
$\gafp^3+\gbfp^3 = \glfp^3-3\glfp$ and $\gafp^2+\gbfp^2=\glfp^2-2$.
Writing $(p+1-\glfp\sqrt{p})^{-1}$ as
$(p+1)^{-1}\left(1-\frac{\glfp\sqrt{p}}{p+1}\right)^{-1}$, using the
geometric series formula and collecting terms, we find \bea M_3(p) &
\ = \ & \frac{2}{p(p+1)} - \frac{\sqrt{p}(3p+1)\glfp}{p(p+1)^2} -
\frac{(p^2+3p+1)\glfp^2}{p(p+1)^3} + \sum_{r=3}^\infty
\frac{p^{r/2}(p-1) \glfp^r}{(p+1)^{r+1}}.\nonumber\\ \eea We use
\eqref{eq:taylorexpandhphi} to replace $\hphi(\log p/\log R)$ in
\eqref{eq:expSFearly1} with $\hphi(0) + O(1/\log^2 R)$ and the above
expansion for $M_3(p)$; the proof is then completed by simple
algebra and recalling the definitions of $A_{r,\mathcal{F}}(p)$ and
$A_{r,\mathcal{F}}'(p)$, \eqref{eq:defArfpArfprimep}.
\end{proof}

\subsection{Formulas for the $r \ge 3$ Terms}

For many families we either know or conjecture a distribution for
the (weighted) Fourier coefficients. If this were the case, then we
could replace the $A_{r,\F}(p)$ with the $r$\textsuperscript{th}
moment. In many applications (for example, using the Petersson
formula for families of cuspidal newforms of fixed weight and
square-free level tending to infinity) we know the moments up to a
negligible correction.

In all the cases we study, the known or conjectured distribution is
even, and the moments have a tractable generating function. Thus we
may show

\begin{lem} Assume for $r \ge 3$ that \be \twocase{A_{r,\F}(p) \ = \ }{M_\ell +
O\left(\frac1{\log^2 R}\right)}{if $r = 2\ell$}{O\left(\frac1{\log^2
R}\right)}{otherwise,}\ee and that there is a nice function $g_M$
such that \be g_M(x) \ = \ M_2 x^2 + M_3 x^3 + \cdots \ = \
\sum_{\ell=2}^\infty M_\ell\ x^\ell. \ee Then the contribution from
the $r \ge 3$ terms in Theorem \ref{thm:expSF} is \be
-\frac{2\hphi(0)}{\log R} \sum_p g_M\left( \frac{p}{(p+1)^2}\right)
\cdot \frac{(p-1)\log p}{p+1} + O\left(\frac1{\log^3 R}\right). \ee
\end{lem}

\begin{proof}
The big-Oh term in $A_{r,\F}(p)$ yields an error of size $1/\log^3
R$. The contribution from the $r \ge 3$ terms in Theorem
\ref{thm:expSF} may therefore be written as \be
-\frac{2\hphi(0)}{\log R} \sum_p \frac{(p-1)\log p}{p+1}
\sum_{\ell=2}^\infty M_\ell \cdot
\left(\frac{p}{(p+1)^2}\right)^\ell + O\left(\frac1{\log^3
R}\right). \ee The result now follows by using the generating
function $g_M$ to evaluate the $\ell$-sum.
\end{proof}

\begin{rek} In the above lemma, note that $g_M(x)$ has even and odd powers of $x$, even though the known or conjectured distribution is even. This is because the expansion in Theorem \ref{thm:expSF} involves $p^{r/2}$, and the only contribution is when $r = 2\ell$. \end{rek}

\begin{lem}\label{lem:STCMconstants}
If the distribution of the weighted Fourier coefficients satisfies
Sato-Tate (normalized to be a semi-circle) with errors in the
moments of size $O(1/\log^2 R)$, then the contribution from the
$r\ge 3$ terms in Theorem \ref{thm:expSF} is \be\label{eq:STsums}
-\frac{2\gamma_{{\rm ST};\widetilde{\mathcal{A}}}\ \hphi(0)}{\log R}
+ O\left(\frac1{\log^3 R}\right), \ee where \be \gamma_{{\rm
ST};\widetilde{\mathcal{A}}} \ = \ \sum_p \frac{(2p+1)(p-1)\log
p}{p(p+1)^3} \ \approx \ .4160714430. \ee If the Fourier
coefficients vanish except for primes congruent to $a \bmod b$
(where $\phi(b) = 2$) and the distribution of the weighted Fourier
coefficients for $p \equiv a \bmod b$ satisfies the analogue of
Sato-Tate for elliptic curves with complex multiplication, then the
contribution from the $r\ge 3$ terms in Theorem \ref{thm:expSF} is
\be -\frac{2\gamma_{{\rm CM},a,b}\ \hphi(0)}{\log R} +
O\left(\frac1{\log^3 R}\right), \ee where \be \gamma_{{\rm CM,a,b}}
\ = \ \sum_{p\equiv a \bmod b}\ \frac{2(3p+1)\log p}{(p+1)^3}. \ee
In particular, \be \gamma_{{\rm CM}1,3} \ \approx \ .38184489, \ \ \
\gamma_{{\rm CM}1,4} \ \approx \ 0.46633061.\ee
\end{lem}

\begin{proof}
If the distribution of the weighted Fourier coefficients satisfies
Sato-Tate (normalized to be a semi-circle here), then $M_\ell =
C_\ell = \frac1{\ell+1}\ncr{2\ell}{\ell}$, the
$\ell$\textsuperscript{th} Catalan number. We have (see sequence A000108 in \cite{Sl}) \bea g_{{\rm
ST}}(x) &  \ = \ & \frac{1-\sqrt{1-4x}}{2x}-1-x \ = \ 2x^2+5x^3 +
14x^4 + \cdots \ = \ \sum_{\ell=2}^\infty C_\ell\ x^\ell \nonumber\\
g_{{\rm ST}}\left(\frac{p}{(p+1)^2}\right) &=&
\frac{2p+1}{p(p+1)^2}. \eea The value for $\gamma_{{\rm
ST};\widetilde{\mathcal{A}}}$ was obtained by summing the
contributions from the first million primes.

For curves with complex multiplication, $M_\ell = D_\ell = 2 \cdot
\foh \ncr{2\ell}{\ell}$; while the actual sequence is just $
\ncr{2\ell}{\ell} = (\ell+1) C_\ell$, we prefer to write it this way
as the first $2$ emphasizes that the contribution is zero for half
the primes, and it is $\foh \ncr{2\ell}{\ell}$ that is the natural
sequences to study. The generating function is \bea g_{{\rm CM}}(x)
& \ = \ & \frac{1-\sqrt{1-4x}}{\sqrt{1-4x}}- 2x \ = \ 6x^2 + 20x^3 +
70x^4
+ \cdots \ = \ \sum_{\ell=2}^\infty D_\ell\ x^\ell \nonumber\\
g_{{\rm CM}}\left(\frac{p}{(p+1)^2}\right) &=&
\frac{2(3p+1)}{(p-1)(p+1)^2}; \eea these numbers are the convolution of the Catalan
numbers and the central binomial (see sequence A000984 in \cite{Sl}). The numerical values were
obtained by calculating the contribution from the first million
primes.
\end{proof}

\begin{rek} It is interesting how close the three sums are. Part of
this is due to the fact that these sums converge rapidly. As the
small primes contribute more to these sums, it is not surprising
that $\gamma_{{\rm CM}1,4} > \gamma_{{\rm CM}1,3}$ (the first primes
for $\gamma_{{\rm CM}1,4}$ are $5$ and $11$, versus $7$ and $13$ for
$\gamma_{{\rm CM}1,3}$).
\end{rek}

\begin{rek} When we investigate one-parameter families of elliptic
curves over $\Q(T)$, it is implausible to assume that for each $p$
the $r$\textsuperscript{th} moment agrees with the
$r$\textsuperscript{th} moment of the limiting distribution up to
negligible terms. This is because there are at most $p$ data points
involved in the weighted averages $A_{r,\F}(p)$; however, it is
enlightening to compare the contribution from the $r \ge 3$ terms in
these families to the theoretical predictions when we have
instantaneous convergence to the limiting distribution.
\end{rek}

We conclude by sketching the argument for identifying the presence
of the Sato-Tate distribution for weight $k$ cuspidal newforms of
square-free level $N \to \infty$. 
In the expansion of $\glf(p)^r$, to first order all that often
matters is the constant term; by the Petersson formula this is the
case for cuspidal newforms of weight $k$ and square-free level
$N\to\infty$, though this is not the case for families of elliptic
curves with complex multiplication. If $r$ is odd then the constant
term is zero, and thus to first order (in the Petersson formula)
these terms do not contribute. For $r=2\ell$ even, the constant term
is $\frac1{\ell+1}\ncr{2\ell}{\ell}$ $=$
$\frac{(2\ell)!}{\ell!(\ell+1)!}$ $=$ $C_\ell$, the
$\ell$\textsuperscript{th} Catalan number. We shall write \be
\glf(p)^r \ = \ \sum_{k=0}^{r/2} b_{r,r-2k} \glf(p^{r-2k}), \ee and
note that if $r=2\ell$ then the constant term is $b_{2\ell,0} =
C_\ell$. We have \bea A_{r,\F}(p) & \ = \ & \fwf\sum_{f\in\F \atop
f\in S(p)} w_R(f) \glf(p)^r \nonumber\\ & = & \fwf\sum_{f\in\F \atop
f\in S(p)} w_R(f) \sum_{k=0}^{r/2} b_{r,r-2k} \glf(p^{r-2k}) \ = \
\sum_{k=0}^{r/2} b_{r,r-2k} A_{r,\F;k}(p), \eea where \be
A_{r,\F;k}(p) \ = \ \fwf\sum_{f\in\F \atop f\in
S(p)}w_R(f)\glf(p^{r-2k}).\ee We expect the main term to be
$A_{2\ell,\F;0}$, which yields the contribution described in
\eqref{eq:STsums}.


\section{Families of cuspidal newforms}\label{sec:cuspnewS}
\setcounter{equation}{0}

Let $\F$ be a family of cuspidal newforms of weight $k$ and prime
level $N$; perhaps we split by sign (the answer is the same,
regardless of whether or not we split). We consider the lower order
correction terms in the limit as $N\to\infty$.

\subsection{Weights}\label{sec:harmweights}

Let \bea \zeta_N(s) & \ = \ & \sum_{n|N^\infty} \frac1{n^s} \ = \
\prod_{p|N}\left(1 - \frac1{p^s}\right)^{-1} \nonumber\\ Z(s,f) &=&
\sum_{n=1}^\infty \frac{\glf(n^2)}{n^s} \ = \
\frac{\zeta_N(s)L(s,f\otimes f)}{\zeta(s)};\eea note \be L(s,{\rm
sym}^2 f) \ = \ \frac{\zeta(2s)Z(s,f)}{\zeta_N(2s)}, \ \ \ Z(1,f) \
= \ \frac{\zeta_N(2)}{\zeta(2)} L(1,{\rm sym}^2 f). \ee To simplify
the presentation, we use the harmonic weights\footnote{The harmonic
weights are essentially constant. By \cite{I1,HL} they can fluctuate
within the family as \be N^{-1-\gep}\ \ll_k \ \omega_R(f) \ \ll_k \
N^{-1+\gep}; \ee if we allow ineffective constants we can replace
$N^\gep$ with $\log N$ for $N$ large.} \be w_R(f) \ = \ \zeta_N(2) /
Z(1,f) \ = \ \zeta(2) / L(1,{\rm sym}^2 f), \ee and note that \be
W_R(\F) \ = \ \sum_{f \in \hkn} w_R(f) \ = \ \frac{(k-1)N}{12} +
O(N^{-1});\ee we may take $R$ to be the analytic conductor, so $R = 15N/ 64\pi^2$. We have introduced the harmonic weights to facilitate
applying the Petersson formula to calculate the average moments
$A_{r,\F}(p)$ from studying $A_{r,\F;k}(p)$. The Petersson formula
(see Corollary 2.10, Equation (2.58) of \cite{ILS}) yields, for $m,
n > 1$ relatively prime to the level $N$,
\be\label{eq:PeterssonFormula} \fwf\sum_{f \in \hkn} w_R(f)
\glf(m)\glf(n) \ = \ \delta_{mn} \ + \ O\left((mn)^{1/4}\frac{\log
2mnN}{k^{5/6} N}\right),\ee where $\delta_{mn} = 1$ if $m=n$ and $0$
otherwise.

\subsection{Results}

From Theorem \ref{thm:expSF}, there are five terms to analyze:
$S_{A'}(\F)$, $S_0(\F)$, $S_1(\F)$, $S_2(\F)$ and $S_A(\F)$. One
advantage of our approach (replacing sums of $\gafp^r + \gbfp^r$
with moments of $\glf(p)^r$) is that the Fourier coefficients of a
generic cuspidal newform should follow Sato-Tate; the Petersson
formula easily gives Sato-Tate on average as we vary the forms while
letting the level tend to infinity, which is all we need here. Thus
$A_{r,\F}(p)$ is basically the $r$\textsuperscript{th} moment of the
Sato-Tate distribution (which, because of our normalizations, is a
semi-circle here). The odd moments of the semi-circle are zero, and
the $(2\ell)$\textsuperscript{th} moment is $C_\ell$. If we let \be
P(\ell) \ = \ \sum_p \frac{(p-1)\log p}{p+1}
\left(\frac{p}{(p+1)^2}\right)^\ell, \ee then we find \be
S_{A,0}(\F) \ = \ -\frac{2\hphi(0)}{\log R} \sum_{\ell=2}^\infty
C_\ell P(\ell), \ee and we are writing the correction term as a
weighted sum of the expected main term of the moments of the Fourier
coefficients; see Lemma \ref{lem:STCMconstants} for another way of
writing this correction. These expansions facilitate comparison with
other families where the coefficients do not follow the Sato-Tate
distribution (such as one-parameter families of elliptic curves with
complex multiplication).

Below we sketch an analysis of the lower order correction terms of
size $1/\log R$ to families of cuspidal newforms of weight $k$ and
prime level $N\to\infty$. We analyze the five terms in the expansion
of $S(\F)$ in Theorem \ref{thm:expSF}.

The following lemma is useful for evaluating many of the sums that
arise. We approximated $\gamma_{{\rm PNT}}$ below by using the first
million primes (see Remark \ref{rek:finch} for an alternate, more
accurate expression for $\gamma_{{\rm PNT}}$). The proof is a
consequence of the prime number theorem; see Section 8.1 of
\cite{Yo1} for details.

\begin{lem}\label{lem:PNTsum2phi}
Let $\theta(t) = \sum_{p \le t} \log p$ and $E(t) = \theta(t)-t$. If
$\hphi$ is a compactly support even Schwartz test function, then \be
\sum_p \frac{2\log p}{p\log R} \hphi\left(2\frac{\log p}{\log
R}\right) \ = \ \frac{\phi(0)}2 + \frac{2\hphi(0)}{\log
R}\left(1+\int_1^\infty \frac{E(t)}{t^2}\ dt\right) +
O\left(\frac1{\log^3 R}\right), \ee where \be\label{eq:defngamma02}
\gamma_{{\rm PNT}} \ = \ 1+\int_1^\infty \frac{E(t)}{t^2}\ dt \
\approx \
-1.33258. \ee 
\end{lem}

\begin{rek}
The constant $\gamma_{{\rm PNT}}$ also occurs in the definition of
the constants $c_{4,1}$ and $c_{4,2}$ in \cite{Yo1}, which arise
from calculating lower order terms in two-parameter families of
elliptic curves. The constants $c_{4,1}$ and $c_{4,2}$ are in error,
as the value of $\gamma_{{\rm PNT}}$ used in \cite{Yo1} double
counted the $+1$. \end{rek}

\begin{rek}\label{rek:finch} Steven Finch has informed us that $\gamma_{{\rm PNT}} =
-\gamma - \sum (\log p)/(p^2-p)$; see\hfill\\
\texttt{http://www.research.att.com/$\sim$njas/sequences/A083343}
for a high precision evaluation and \cite{Lan,RoSc} for proofs.
\end{rek}

\begin{thm}\label{thm:cuspnewformslot} Let $\hphi$ be supported in $(-\sigma, \sigma)$ for some
$\sigma < 4/3$ and consider the harmonic weights \be w_R(f) \ = \
\zeta(2) / L(1,{\rm sym}^2 f). \ee  Then \bea S(\F) & \ = \ &
\frac{\phi(0)}2 + \frac{2(-\gamma_{{\rm ST};0}\ + \gamma_{{\rm
ST};2}-\gamma_{{\rm ST};\widetilde{\mathcal{A}}}\ +\gamma_{{\rm
PNT}})\hphi(0)}{\log R} + O\left(\frac1{\log^3 R}\right)\eea where
\bea\label{eq:satotatedistrconstants}
\begin{array}{ccccr}
\gamma_{{\rm ST};0} & \ = \ & \sum_p \frac{2\log p} {p(p+1)} & \
\approx \
& 0.7691106216 \\ \ \\
\gamma_{{\rm ST};2} & = & \sum_p \frac{(4p^2+3p+1)\log p}{p(p+1)^3}
& \ \approx \ & 1.1851820642 \\ \ \\ \gamma_{{\rm
ST};\widetilde{\mathcal{A}}}\  & \
= \ & \sum_{\ell=2}^\infty C_\ell P(\ell) & \ \approx \ & 0.4160714430\\
\ \\ \gamma_{{\rm PNT}} & = & 1 + \int_1^\infty \frac{E(t)}{t^2}\ dt
& \ \approx \ & -1.33258
\end{array}\eea and \be -\gamma_{{\rm ST};0}\ + \gamma_{{\rm ST};2}\ -\gamma_{{\rm ST};\widetilde{\mathcal{A}}} \ = \ 0.
\ee
\end{thm}

The notation above is to emphasize that these coefficients arise
from the Sato-Tate distribution. The subscript $0$ (resp. $2$)
indicates that this contribution arises from the $A_{0,\F}(p)$
(resp. $A_{2,\F}(p)$) terms, the subscript $\widetilde{\mathcal{A}}$
indicates the contribution from
$S_{\widetilde{\mathcal{A}}}(\mathcal{F})$ (the $A_{r,\F}(p)$ terms
with $r\ge 3$), and we use PNT for the final constant to indicate a
contribution from applying the Prime Number Theorem to evaluate sums
of our test function.

\begin{proof} The proof follows by calculating the contribution of
the five pieces in Theorem \ref{thm:expSF}. We assume $\hphi$ is an
even Schwartz function such that $\supp(\hphi) \subset (-\sigma,
\sigma)$, with $\sigma < 4/3$, $\F$ is the family of weight $k$ and
prime level $N$ cuspidal newforms (with $N \to \infty)$, and we use
the harmonic weights of \S\ref{sec:harmweights}. Straightforward
algebra shows\footnote{Except for the $S_A(\F)$ piece, where a
little care is required; see Appendix
\ref{sec:evalSAFcuspnewprimelevelN} for details.} \ben

\item $S_{A'}(\F) \ll N^{-1/2}$.

\item $S_{A}(\F)=  -\frac{2\gamma_{{\rm
ST};\widetilde{\mathcal{A}}}\ \hphi(0)}{\log R}
+O\left(\frac1{R^{.11}\log^2 R}\right) \ +\ O\left(\frac{\log
R}{N^{.73}}\right) + O\left(\frac{N^{3\sigma/4}\log R}{N}\right)$.
In particular, for test functions supported in $(-4/3,4/3)$ we have
$S_{A}(\F) = -\frac{2\gamma_{{\rm ST};\widetilde{\mathcal{A}}}\
\hphi(0)}{\log R} + O\left(R^{-\gep}\right)$, where $\gamma_{{\rm
ST};\widetilde{\mathcal{A}}}$ $\approx$ $.4160714430$ (see Lemma
\ref{lem:STCMconstants}).

\item $S_0(\F) = \phi(0) + \frac{2(2\gamma_{{\rm PNT}} -
\gamma_{{\rm ST};0})\hphi(0)}{\log R} + O\left(\frac1{\log^3
R}\right)$, where $\gamma_{{\rm ST};0} =
 \sum_p \frac{2\log p} {p(p+1)} \ \approx \ 0.7691106216$,
$\gamma_{{\rm PNT}}$ $=$  $1 + \int_1^\infty \frac{E(t)}{t^2}\ dt
\approx -1.33258$.

\item $S_1(\F)  \ll \frac{\log N}{N}\sum_{p=2}^{R^\sigma}
\frac{p^{1/4}}{p^{1/2}} \ll  N^{\frac34\sigma - 1} \log N$.

\item Assume $\sigma < 4$. Then \bea S_2(\F) & \ =\ & -\frac{\phi(0)}2
- \frac{2\gamma_{{\rm PNT}}\ \hphi(0)}{\log R}+ \frac{\gamma_{{\rm
ST};2}\ \hphi(0)}{\log R} + O\left(\frac1{\log^3 R}\right),
\nonumber\\ \gamma_{{\rm ST};2} & = & \sum_p \frac{(4p^2+3p+1)\log
p}{p(p+1)^3} \ \approx \ 1.1851820642 \eea and $\gamma_{{\rm PNT}}$
is defined in \eqref{eq:defngamma02}.

\een

 The $S_{A'}(\F)$ piece
does not contribute, and the other four pieces contribute multiples
of $\gamma_{{\rm ST};0}$, $\gamma_{{\rm ST};2}$, $\gamma_{{\rm
ST};3}$  and $\gamma_{{\rm PNT}}$.
\end{proof}

\begin{rek} Numerical calculations will never suffice
to show that $-\gamma_{{\rm ST};1}\ + \gamma_{{\rm
ST};2}-\gamma_{{\rm ST};\widetilde{\mathcal{A}}}$ is exactly zero;
however, we have \bea
 -\gamma_{{\rm ST};0}\ + \gamma_{{\rm ST};2}-\gamma_{{\rm ST};\widetilde{\mathcal{A}}} & \ = \ & \sum_p
\left(- \frac2{p(p+1)}+
\frac{4p^2+3p+1}{p(p+1)^3}-\frac{(2p+1)(p-1)}{p(p+1)^3} \right)\log
p \nonumber\\ & \ = \ & \sum_p 0 \cdot \log p \ = \ 0. \eea This may
also be seen by calculating the lower order terms using a different
variant of the explicit formula. Instead of expanding in terms of
$\gafp^m+\gbfp^m$ we expand in terms of $\glf(p^m)$. The terms which
depend on the Fourier coefficients are given by \bea -2  \sum_{p|N}
\sum_{m=1}^\infty \fwf\sum_{f \in \hkn} w_R(f)\frac{\glf(p)^m\log
p}{p^{m/2}\log R}\ \hphi\left(m\frac{\log p}{\log R}\right) +
2\sum_{p\notdiv N} \frac{\log p}{p\log R} \hphi\left(2\frac{\log
p}{\log R}\right) \nonumber\\  -2\ \sum_{p\notdiv N}
\sum_{m=1}^\infty \fwf\sum_{f \in \hkn} w_R(f)\frac{\glf(p^m)\log
p}{p^{m/2}\log R} \left( \hphi\left(m\frac{\log p}{\log R}\right) -
\frac1p\hphi\left((m+2)\frac{\log p}{\log
R}\right)\right);\nonumber\\ \eea this follows from trivially
modifying Proposition 2.1 of \cite{Yo1}. For $N$ a prime, the
Petersson formula shows that only the second piece contributes for
$\sigma < 4/3$, and we regain our result that the lower order term
of size $1/\log R$ from the Fourier coefficients is just
$2\gamma_{{\rm PNT}}\hphi(0)/\log R$. We prefer our expanded version
as it shows how the moments of the Fourier coefficients at the
primes influence the correction terms, and will be useful for
comparisons with families that either do not satisfy Sato-Tate, or
do not immediately satisfy Sato-Tate with negligible error for each
prime.
\end{rek}


\section{Preliminaries for Families of Elliptic
Curves}\label{sec:prelimECfam} \setcounter{equation}{0}

\subsection{Notation}

We review some notation and results for elliptic curves; see
\cite{Kn,Si1,Si2} for more details. Consider a one-parameter family
of elliptic curves over $\Q(T)$: \be \mathcal{E}:\ \ y^2 \ = \ x^3 + A(T)x + B(T),
\ \ A(T), B(T) \in \Z[T]. \ee For each $t \in \Z$ we obtain an
elliptic curve $E_t$ by specializing $T$ to $t$. We denote the
Fourier coefficients by $a_t(p) = \lambda_t(p)\sqrt{p}$; by Hasse's
bound we have $|a_t(p)| \le 2\sqrt{p}$ or $|\lambda_t(p)| \le 2$.
The discriminant and $j$-invariant of the elliptic curve $E_t$ are
\be \Delta(t) \ = \ -16(4 A(t)^3 + 27 B(t)^2), \ \ \ j(t) \ = \
-1728 \cdot 4 A(t)^3 / \Delta(t). \ee

Consider an elliptic curve $y^2 = x^3 + Ax + B$ (with $A, B\in \Z$)
and a prime $p \ge 5$.
As $p \ge 5$, the equation is minimal if either $p^4$ does not
divide $A$ or $p^6$ does not divide $B$. If the equation is minimal
at $p$ then \be\label{eq:expforatp} a_t(p) \ = \ -\sum_{x \bmod p}
\js{x^3 + A(t)x+B(t)}\ = \ p+1 - N_t(p), \ee where $N_t(p)$ is the
number of points (including infinity) on the reduced curve $\tilde E
\bmod p$. Note that $a_{t+mp}(p) = a_t(p)$. This periodicity is our
analogue of the Petersson formula; while it is significantly weaker,
it will allow us to obtain results for sufficiently small support.

Let $E$ be an elliptic curve with minimal Weierstrass equation at
$p$, and assume $p$ divides the discriminant (so the reduced curve
modulo $p$ is singular). Then $a_E(p) \in \{-1,0,1\}$, depending on
the type of reduction. By changing coordinates we may write the
reduced curve as $(y-\alpha x)(y-\beta x) = x^3$. If $\alpha =
\beta$ then we say $E$ has a cusp and additive (or unstable)
reduction at $p$, and $a_E(p) = 0$. If $\alpha \neq \beta$ then $E$
has a node and multiplicative (or semi-stable) reduction at $p$; if
$\alpha, \beta \in \Q$ we say $E$ has split reduction and $a_E(p) =
1$, otherwise it has non-split reduction and $a_E(p) = -1$. We shall
see later that many of our arguments are simpler when there is no
multiplicative reduction, which is true for families with complex
multiplication.

Our arguments below are complicated by the fact that for many $p$
there are $t$ such that $y^2 = x^3 + A(T)x + B(T)$ is not minimal at
$p$ when we specialize $T$ to $t$. For the families we study, the
specialized curve at $T=t$ is minimal at $p$ provided $p^k$ ($k$
depends on the family) does not divide a polynomial $D(t)$ (which
also depends on the family, and is the product of irreducible
polynomial factors of $\Delta(t)$). For example, we shall later
study the family with complex multiplication \be y^2 \ = \ x^3 +
B(6T+1)^\kappa, \ee where $B|6^\infty$ (i.e., $p|B$ implies $p$ is
$2$ or $3$) and $\kappa \in \{1,2\}$). Up to powers of $2$ and $3$,
the discriminant is $\Delta(T) = (6T+1)^{2\kappa}$, and note that
$(6t+1,6) = 1$ for all $t$. Thus for a given $t$ the equation is
minimal for all primes provided that $6t+1$ is sixth-power free if
$\kappa = 1$ and cube-free if $\kappa = 2$. In this case we would
take $D(t) = 6t+1$ and $k=6/\kappa$. To simplify the arguments, we
shall sieve our families, and rather than taking all $t \in [N,2N]$
instead additionally require that $D(t)$ is $k$\textsuperscript{th}
power free. Equivalently, we may take all $t \in [N,2N]$ and set the
weights to be zero if $D(t)$ is not $k$\textsuperscript{th} power
free. Thus throughout the paper we adopt the following conventions:
\bi
\item the family is $y^2 = x^3 + A(T)x + B(T)$ with $A(T), B(T) \in
\Z[T]$, and we specialize $T$ to $t\in [N,2N]$ with $N\to\infty$;
\item we associate polynomials $D_1(T), \dots, D_d(T)$ and
integers\label{page:conventionsk} $k_1, \dots, k_d \ge 3$, and the
weights are $w_R(t) = 1$ if $t \in [N,2N]$ and $D_i(t)$ is
$k_i$\textsuperscript{th} power free, and $0$ otherwise;
\item $\log R$ is the average log-conductor of the family, and $\log  R = (1+o(1))\log N$ (see \cite{DM2,Mil2}).\ei

\subsection{Sieving}

For ease of notation, we assume that we have a family where $D(T)$
is an irreducible polynomial, and thus there is only one power, say
$k$; the more general case proceeds analogously. \emph{We assume
that $k \ge 3$ so that certain sums are small (if $k \le 2$ we need
to assume either the ABC of Square-Free Sieve Conjecture).} Let
$\delta^k N^d$ exceed the largest value of $|D(t)|$ for
$t\in[N,2N]$. We say a $t\in [N,2N]$ is \textbf{good} if $D(t)$ is
$k$\textsuperscript{th} power free; otherwise we say $t$ is
\textbf{bad}. To determine the lower order correction terms we must
evaluate $S(\F)$, which is defined in \eqref{eq:defSF}. We may write
\be S(\F) \ = \ \frac1{W_R(\F)} \sum_{t = N}^{2N} w_R(t) S(t). \ee
As $w_R(t) = 0$ if $t$ is bad, for bad $t$ we have the freedom of
defining $S(t)$ in any manner we may choose. Thus, even though the
expansion for $a_t(p)$ in \eqref{eq:expforatp} requires the elliptic
curve $E_t$ to be minimal at $p$, we may use this definition for all
$t$. We use inclusion - exclusion to write our sums in a more
tractable form; the decomposition is standard (see, for example,
\cite{Mil2}). Letting $\ell$ be an integer (its size will depend on
$d$ and $k$), we have \bea\label{eq:sfmobiusexpand1} S(\F) & \ = \ &
\frac1{W_R(\F)} \sum_{t = N \atop D(t)\ k-{\rm power\ free}}^{2N}
w_R(t) S(t) \nonumber\\ & = & \frac1{W_R(\F)} \sum_{d=1}^{\log^\ell
N} \mu(d)\sum_{t = N \atop D(t) \equiv 0 \bmod d^k}^{2N} S(t) +
\frac1{W_R(\F)} \sum_{d=1+\log^\ell N}^{\delta N^{d/k}}\mu(d)
\sum_{t = N \atop D(t) \equiv 0 \bmod d^k}^{2N} S(t), \nonumber\\
\eea where $\mu$ is the M$\ddot{{\rm o}}$bius function. For many families we can show that \be \sum_{t = N \atop D(t) \equiv
0 \bmod d^k}^{2N} S(t)^2 \ = \ O\left(\frac{N}{d^k}\right). \ee If
this condition\footnote{Actually, this condition is a little
difficult to use in practice. It is easier to first pull out the sum
over all primes $p$ and then square; see \cite{Mil2} for details.}
holds, then applying the Cauchy-Schwarz inequality to
\eqref{eq:sfmobiusexpand1} yields \bea\label{eq:sfmobiusexpand2}
S(\F) & \ = \ & \frac1{W_R(\F)} \sum_{d=1}^{\log^\ell N}
\mu(d)\sum_{t = N \atop D(t) \equiv 0 \bmod d^k}^{2N} S(t) +
O\left(\frac1{W_R(\F)} \sum_{d=1+\log^\ell N}^{\delta N^{d/k}}
\sqrt{\frac{N}{d^k}} \cdot \sqrt{N} \right) \nonumber\\ &=&
\frac1{W_R(\F)} \sum_{d=1}^{\log^\ell N} \mu(d)\sum_{t = N \atop
D(t) \equiv 0 \bmod d^k}^{2N} S(t) + O\left(\frac{N}{W_R(\F)} \cdot
(\log N)^{-(\foh k - 1) \cdot\ell}\right). \eea For all our families
$W_R(\F)$ will be of size $N$ (see \cite{Mil2} for a proof). Thus
for $\ell$ sufficiently large the error term is significantly
smaller than $1/\log^3 R$, and hence negligible (remember $\log R = (1+o(1))\log N$). Note it is
important that $k \ge 3$, as otherwise we would have obtained $\log
N$ to a non-negative power (as we would have summed $1/d$). For
smaller $k$ we may argue by using the ABC or Square-Free Sieve
Conjectures.

The advantage of the above decomposition is that the sums are over
$t$ in arithmetic progressions, and we may exploit the relation
$a_{t+mp}(p) = a_t(p)$ to determine the family averages by
evaluating sums of Legendre symbols. This is our analogue, poor as
it may be, to the Petersson formula.

There is one technicality that arises here which did not in
\cite{Mil2}. There the goal was only to calculate the main term in
the $n$-level densities; thus ``small'' primes ($p$ less than a
power of $\log N$) could safely be ignored. If we fix a $d$ and
consider all $t$ with $D(t) \equiv 0 \bmod d^k$, we obtain a union
of arithmetic progressions, with each progression having step size
$d^k$. We would like to say that we basically have $(N/d^k)/p$
complete sums for each progression, with summands $a_{t_0}(p),
a_{t_0+d^kp}(p), a_{t_0+2d^kp}(p)$, and so on. The problem is that
if $p|d$ then we do not have a complete sum, but rather we have the
same term each time! We discuss how to handle this obstruction in
the next sub-section.

\subsection{Moments of the Fourier Coefficients and the Explicit Formula}

Our definitions imply that $A_{r,\F}(p)$ is obtained by averaging
$\lambda_t(p)^r$ over all $t \in [N,2N]$ such that $p \ \notdiv
\Delta(t)$; the remaining $t$ yield $A_{r,\F}'(p)$. We have
sums such as \be \frac1{W_R(\F)} \sum_{d=1}^{\log^\ell N}
\mu(d)\sum_{t = N \atop D(t) \equiv 0 \bmod d^k}^{2N} S(t).\ee In
all of our families $D(T)$ will be the product of the irreducible
polynomial factors of $\Delta(T)$. For ease of exposition, we assume
$D(T)$ is given by just one factor.

We expand $S(\F)$ and $S(t)$ by using Theorem \ref{thm:expSF}. The
sum of $S(t)$ over $t$ with $D(t) \equiv 0 \bmod d^k$ breaks up into
two types of sums, those where $\Delta(t) \equiv 0 \bmod p$ and
those where $\Delta(t) \not\equiv 0 \bmod p$. For a fixed $d$, the
goal is to use the periodicity of the $t$-sums to replace
$A_{r,\F}(p)$ with complete sums.

Thus we need to understand complete sums. If $t \in [N,2N]$, $d \le
\log^\ell N$ and $p$ is fixed, then the set of $t$ such that $D(t)
\equiv 0 \bmod d^k$ is a union of arithmetic progressions; the
number of arithmetic progressions equals the number of distinct
solutions to $D(t) \equiv 0 \bmod d^k$, which we denote by
$\nu_D(d^k)$. We have $(N/d^k)/p$ complete sums, and at most
$p$ summands left over.

Recall \bea A_{r,\F}(p)  \ = \  \fwf\sum_{f \in \F \atop f \in S(p)}
w_R(f)\glfp^r,\ \ \ A_{r,\F}'(p)  \ = \  \fwf\sum_{f \in \F \atop f
\not\in S(p)} w_R(f)\glfp^r, \eea and set \be \mathcal{A}_{r,\F}(p)
\ = \ \sum_{t \bmod p \atop p\ \notdiv \Delta(t)} a_t(p)^r \ = \
p^{r/2}\sum_{t \bmod p \atop p\ \notdiv \Delta(t)} \lambda_t(p)^r,\
\ \ \ \ \ \ \ \ \ \mathcal{A}_{r,\F}'(p) \ = \ \sum_{t \bmod p \atop
p|\Delta(t)} a_t(p)^r. \ee

\begin{lem}\label{lem:expandArfpwithnuD} Let $D$ be a product of
irreducible polynomials such that (i) for all $t$ no two factors are
divisible by the same prime; (ii) the same $k \ge 3$ (see the
conventions on page \pageref{page:conventionsk}) is associated to
each polynomial factor. For any $\ell \ge 7$ we have \bea
A_{r,\F}(p) &\ = \ & \frac{\mathcal{A}_{r,\F}(p)}{p\cdot
p^{r/2}}\left[1 + \frac{\nu_D(p^k)}{p^{k}} \left(1 -
\frac{\nu_D(p^k)}{p^k}\right)^{-1}\right] +
O\left(\frac1{\log^{\ell/2} N}\right) \nonumber\\ A_{r,\F}'(p) &\ =
\ & \frac{\mathcal{A}_{r,\F}'(p)}{p\cdot p^{r/2}}\left[1 +
\frac{\nu_D(p^k)}{p^{k}} \left(1 -
\frac{\nu_D(p^k)}{p^k}\right)^{-1}\right] +
O\left(\frac1{\log^{\ell/2} N}\right). \eea
\end{lem}

\begin{proof}
For our family, the $d \ge \log^\ell N$ terms give a negligible
contribution. We rewrite $A_{r,\F}(p)$ as \bea\label{eq:CMexpArfp}
A_{r,\F}(p) & \ = \ & \fwf\sum_{t \in [N,2N], p\ \notdiv D(t) \atop
D(t)\ k-{\rm power\ free}} \lambda_t(p)^r \nonumber\\ &= &
\frac1{W_R(\F)} \sum_{d=1}^{\log^\ell N} \mu(d)\sum_{t \in [N,2N],
p\ \notdiv D(t) \atop
D(t) \equiv 0 \bmod d^k}^{2N} \lambda_t(p)^r + O\left(\log^{-\ell/2} N\right)\nonumber\\
&= & \frac1{W_R(\F)} \sum_{d=1}^{\log^\ell N} \mu(d) \left[
\frac{\nu_D(d^k)N/d^k}{p}\sum_{t \bmod p \atop p \ \notdiv D(t)}
\lambda_t(p)^r \right] + O\left(\frac1{W_R(\F)}
\sum_{d=1}^{\log^\ell N} p2^r \right) \nonumber\\ & & \ \ - \
\frac1{W_R(\F)} \sum_{d=1}^{\log^\ell N} \mu(d)\delta_{p|d} \left[
\frac{\nu_D(d^k)N/d^k}{p}\sum_{t \bmod p \atop p \ \notdiv D(t)}
\lambda_t(p)^r \right],  \eea where $\delta_{p|d} = 1$ if $p|d$ and
0 otherwise. For sufficiently small support the big-Oh term above is
negligible. As $k\ge 3$, we have  \bea W_R(\F) & \ = \ & N \prod_{p}
\left(1 - \frac{\nu_D(d^k)}{p^k}\right) +
O\left(\frac{N}{\log^{\ell/2} N}\right) \nonumber\\ & = & N
\sum_{d=1}^{\log^\ell N} \frac{\mu(d)\nu_D(d^k)}{d^k} +
O\left(\frac{N}{\log^{\ell/2} N}\right). \eea For the terms with
$\mu(d)\delta_{p|d}$ in \eqref{eq:CMexpArfp}, we may write $d$ as
$\tilde{d}p$, with $(\tilde{d},p) = 1$ (the $\mu(d)$ factor forces
$d$ to be square-free, so $p||d$). For sufficiently small support,
\eqref{eq:CMexpArfp} becomes \bea
\frac{\mathcal{A}_{r,\F}(p)}{p\cdot p^{r/2}}\left[1 +
\frac{\nu_D(p^k)}{p^{k}} \left(1 -
\frac{\nu_D(p^k)}{p^k}\right)^{-1}\right] + O\left(\log^{-\ell/2}
N\right); \eea this is because \bea \frac1{W_R(\F)} \sum_{d = 1
\atop p|d}^{\log^\ell N} \frac{\mu(d)\nu_D(d^k)N}{d^k} & \ = \ &
\frac{\mu(p)\nu_D(p^k)}{p^k} \sum_{\tilde d = 1 \atop p\ \notdiv
\tilde d}^{\log^\ell N} \frac{\mu(\tilde d)\nu_D(\tilde
d^k)N}{\tilde d^k} \nonumber\\ &=& -\frac{\nu_D(p^k)}{p^{k}}
\left[\left(1 - \frac{\nu_D(p^k)}{p^k}\right)^{-1} +
O\left(\frac{1}{\log^{\ell/2} N}\right)\right] \eea (the last line
follows because of the multiplicativity of $\nu_D$ (see for example
\cite{Nag}) and the fact that we are missing the factor
corresponding to $p$). The proof for $A_{r,\F}'(p)$ follows
analogously.
\end{proof}

We may rewrite the expansion in Theorem \ref{thm:expSF}. We do not
state the most general version possible, but rather a variant that
will encompass all of our examples.

\begin{thm}[Expansion for $S(\F)$ for many elliptic curve families]\label{thm:expSFECs}
Let $y^2 = x^3 + A(T)x + B(T)$ be a family of elliptic curves over
$\Q(T)$. Let $\Delta(T)$ be the discriminant (and the only primes
dividing the greatest common divisor of the coefficients of
$\Delta(T)$ are 2 or 3), and let $D(T)$ be the product of the
irreducible polynomial factors of $\Delta(T)$. Assume for all $t$
that no prime simultaneously divides two different factors of
$D(t)$, that each specialized curve has additive reduction at $2$
and $3$, and that there is a $k \ge 3$ such that for $p \ge 5$ each
specialized curve is minimal provided that $D(T)$ is
$k$\textsuperscript{{\rm th}} power free (if the equation is a
minimal Weierstrass equation for all $p \ge 5$ we take $k=\infty$);
thus we have the same $k$ for each irreducible polynomial factor of
$D(T)$. Let $\nu_D(d)$ denote the number of solutions to $D(t)
\equiv 0 \bmod d$. Set $w_R(t) = 1$ if $t \in [N,2N]$ and $D(t)$ is
$k$\textsuperscript{{\rm th}} power free, and $0$ otherwise. Let
\bea\label{eq:defHDnup} \mathcal{A}_{r,\F}(p) & \ = \ & \sum_{t
\bmod p \atop p\ \notdiv \Delta(t)} a_t(p)^r \ = \ p^{r/2}\sum_{t
\bmod p \atop p\ \notdiv \Delta(t)} \lambda_t(p)^r, \ \ \
\mathcal{A}_{r,\F}'(p)  \ = \ \sum_{t
\bmod p \atop p| \Delta(t)} a_t(p)^r  \nonumber\\
\widetilde{\mathcal{A}}_\F(p) &\ = \ & \sum_{t \bmod p \atop p
\notdiv \Delta(t)} \frac{a_t(p)^3}{p^{3/2}(p+1 - a_t(p))} \ = \
\sum_{t \bmod p \atop p\ \notdiv \Delta(t)}
\frac{\lambda_t(p)^3}{p+1 - \lambda_t(p)\sqrt{p}} \nonumber\\
H_{D,k}(p) & = & 1 + \frac{\nu_D(p^k)}{p^{k}} \left(1 -
\frac{\nu_D(p^k)}{p^k}\right)^{-1}. \eea

We have \bea\label{eq:expansionsievedECfamSF} & & S(\F)  \ = \
-2\hphi(0) \sum_p \sum_{m=1}^\infty
\frac{\mathcal{A}_{m,\F}'(p)H_{D,k}(p)\log p}{p^{m+1}\log R}
\nonumber\\
& & -2\hphi(0)\sum_p \frac{2\mathcal{A}_{0,\F}(p)H_{D,k}(p)\log p}
{p^2(p+1)\log R}\ +\ 2\sum_p
\frac{2\mathcal{A}_{0,\F}(p)H_{D,k}(p)\log p}{p^2\log R}\
\hphi\left(2\frac{\log p}{\log R}\right)\nonumber\\
& &-2\sum_p \frac{\mathcal{A}_{1,\F}(p)H_{D,k}(p)}{p^2} \frac{\log
p}{\log R}\ \hphipr + 2 \hphi(0)\sum_p
\frac{\mathcal{A}_{1,\F}(p)H_{D,k}(p)
(3p+1)}{p^{2}(p+1)^2}\frac{\log p}{\log R}\nonumber\\ & & -2\sum_p
\frac{\mathcal{A}_{2,\F}(p)H_{D,k}(p)\log p}{p^3\log R}\
\hphi\left(2\frac{\log p}{\log R}\right) + 2\hphi(0)\sum_p
\frac{\mathcal{A}_{2,\F}(p)H_{D,k}(p)(4p^2+3p+1)\log p}{p^3(p+1)^3\log R} \nonumber\\
& & - 2\hphi(0)\sum_{p}
\frac{\widetilde{\mathcal{A}}_\F(p)H_{D,k}(p) p^{3/2}(p-1)\log
p}{p(p+1)^3 \log R} + O\left(\frac1{\log^3 R}\right)\nonumber\\ & =&
S_{\mathcal{A}'}(\F) + S_0(\F)+S_1(\F)+S_2(\F) +
S_{\widetilde{\mathcal{A}}}(\F)+O\left(\frac1{\log^3 R}\right). \eea
If the family only has additive reduction (as is the case for our
examples with complex multiplication), then the
$\mathcal{A}_{m,\F}'(p)$ piece contributes $0$.
\end{thm}

\begin{proof} The proof follows by using Lemma
\ref{lem:expandArfpwithnuD} to simplify Theorem \ref{thm:expSF}, and
\eqref{eq:taylorexpandhphi} to replace the $\hphi(m\log p/\log R)$
terms with $\hphi(0) + O(\log^{-2} R)$ in the
$\mathcal{A}_{m,\F}'(p)$ terms. See Remark
\ref{rek:troubleoneparamecs} for comments on the need to numerically
evaluate the $\widetilde{\mathcal{A}}_\F(p)$ piece.
\end{proof}

For later use, we record a useful variant of Lemma
\ref{lem:PNTsum2phi}.

\begin{lem}\label{lem:PNTsum2phi13}
Let $\varphi$ be the Euler totient function, and \be
\theta_{a,b}(t)\ =\ \sum_{p \le t \atop p \equiv a \bmod b} \log p,
\ \ \ E_{a,b}(t)\ =\ \theta_{a,b}(t)-\frac{t}{\varphi(b)}.\ee If
$\hphi$ is a compactly support even Schwartz test function, then \be
2\sum_p \frac{2\log p}{p\log R} \hphi\left(2\frac{\log p}{\log
R}\right) \ = \ \frac{\phi(0)}2 + \frac{2\hphi(0)}{\log
R}\left(1+\int_1^\infty \frac{2E_{1,3}(t)}{t^2}\ dt\right) +
O\left(\frac1{\log^3 R}\right), \ee where \bea \gamma_{{\rm
PNT};1,3} & \ = \ & 1+\int_1^\infty \frac{2E_{1,3}(t)}{t^2}\ dt \
\approx \ -2.375 \nonumber\\ \gamma_{{\rm PNT};1,4} &=& 1 +
\int_1^\infty \frac{2E_{1,4}(t)}{t^2}\ dt \ \approx \ -2.224; \eea
$\gamma_{{\rm PNT};1,3}$ and $\gamma_{{\rm PNT};1,4}$ were
approximated by integrating up to the four millionth prime,
67,867,979.
\end{lem}

\begin{rek} Steven Finch has informed us that, similar to Remark \ref{rek:finch},
using results from \cite{Lan,Mor} yields formulas for $\gamma_{{\rm
PNT};1,3}$ and $\gamma_{{\rm PNT};1,4}$ which converge more rapidly:
\begin{eqnarray}
\gamma _{\text{PNT};1,3}&\ = \ &-2\gamma -4\log2\pi +\log 3+6\log
\Gamma \left(\frac13\right)-2\sum_{p \equiv 1, 2 \bmod 3}
\frac{\log p}{p^2 - p^{\delta_{1,3}(p)}} \nonumber\\
&\approx &-2.375494 \nonumber\\
\gamma _{\text{PNT};1,4}& \ = \ &-2\gamma -3\log 2\pi+4\log \Gamma
\left(\frac14\right)-2\sum_{p\equiv 1,3 \bmod 4}\frac{\log
p}{p^2-p^{\delta_{1,4}(p)}}\nonumber\\ &\approx &-2.224837;
\end{eqnarray}
here $\gamma$ is Euler's constant and $\delta_{1,n}(p) = 1$ if
$p\equiv 1 \bmod n$ and $0$ otherwise.
\end{rek}


\section{Examples: One-parameter families of elliptic curves over
$\Q(T)$}\label{sec:oneparamfamellcurvesS} \setcounter{equation}{0}

We calculate the lower order correction terms for several
one-parameter families of elliptic curves over $\Q(T)$, and compare
the results to what we would obtain if there was instant convergence
(for each prime $p$) to the limiting distribution of the Fourier
coefficients. We study families with and without complex
multiplication, as well as families with forced torsion points or
rank. We perform the calculations in complete detail for the first
family, and merely highlight the changes for the other families.

\subsection{CM Example: The family $y^2 = x^3 + B (6T+1)^\kappa$ over $\Q(T)$}

\subsubsection{Preliminaries}

Consider the following one-parameter family of elliptic curves over
$\Q(T)$ with complex multiplication: \be y^2 \ = \ x^3 + B
(6T+1)^\kappa, \ \ \ B \in \{1,2,3,6\}, \ \ \ \kappa \in \{1,2\}, \
\ \ k = 6/\kappa. \ee We obtain slightly different behavior for the
lower order correction terms depending on whether or not $B$ is a
perfect square for all primes congruent to 1 modulo 3. For example,
if $B=b^2$ and $\kappa = 2$, then we have forced a torsion point of
order 3 on the elliptic curve over $\Q(T)$, namely $(0,b(6T+1))$.
The advantage of using $6T+1$ instead of $T$ is that $(6T+1,6) = 1$,
and thus we do not need to worry about the troublesome primes $2$
and $3$ (each $a_t(p) = 0$ for $p\in \{2,3\}$). Up to powers of $2$
and $3$ the discriminant is $(6T+1)^\kappa$; thus we take $D(T) =
6T+1$. For each prime $p$ the specialized curve $E_t$ is minimal at
$p$ provided that $p^{2k}\ \notdiv 6t+1$. If $p^{2k}|6t+1$ then
$w_R(t) = 0$, so we may define the summands any way we wish; it is
convenient to use \eqref{eq:expforatp} to define $a_t(p)$, even
though the curve is not minimal at $p$. In particular, this implies
that $a_t(p) = 0$ for any $t$ where $p^3|6t+1$.

One very nice property of our family is that it only has additive
reduction; thus if $p|D(t)$ but $p^{2k}\ \notdiv D(t)$ then $a_t(p)
= 0$. As our weights restrict our family to $D(t)$ being $k =
6/\kappa$ power free, we always use \eqref{eq:expforatp} to define
$a_t(p)$.

It is easy to evaluate $A_{1,\F}(p)$ and $A_{2,\F}(p)$. While these
sums are the average first and second moments over primes \emph{not}
dividing the discriminant, as $a_t(p) = 0$ for $p|\Delta(t)$ we may
extend these sums to be over all primes.

We use Theorem \ref{thm:expSFECs} to write the 1-level density in a
tractable manner. Straightforward calculation (see Appendix
\ref{sec:calc9tp1fam} for details) shows that \bea
\twocase{\mathcal{A}_{0,\F}(p)
& \ = \ & }{p-1}{if $p \ge 5$}{0}{otherwise}\nonumber\\
\mathcal{A}_{1,\F}(p) & \ = \ & 0 \nonumber\\
\twocase{\mathcal{A}_{2,\F}(p) & \ = \ & }{2p^2 - 2p}{if $p \equiv 1
\bmod 3$}{0}{otherwise.} \eea Not surprisingly, neither the zeroth,
first or second moments depend on $B$ or on $\kappa$; this
universality leads to the common behavior of the main terms in the
$n$-level densities. We shall see dependence on the parameters $B$
and $\kappa$ in the higher moments $\mathcal{A}_{r,\mathcal{F}}(p)$,
and this will lead to different lower order terms for the different
families.

As we are using Theorem \ref{thm:expSFECs} instead of Theorem
\ref{thm:expSF}, each prime sum is weighted by \be H_{D,k}(p) \ = \
1 + \frac{\nu_D(p^k)}{p^{k}} \left(1 -
\frac{\nu_D(p^k)}{p^k}\right)^{-1} \ = \ H_{D,k}^{{\rm main}}(p) +
H_{D,k}^{{\rm sieve}}(p), \ee with $H_{D,k}^{{\rm main}}(p) = 1$.
$H_{D,k}^{{\rm sieve}}(p)$ arises from sieving our family to $D(t)$
being $(6/\kappa)$-power free. We shall calculate the contribution
of these two pieces separately. We expect the contribution from
$H_{D,k}^{{\rm sieve}}(p)$ to be significantly smaller, as each
$p$-sum is decreased by approximately $1/p^k$.

\subsubsection{Contribution from $H_{D,k}^{{\rm main}}(p)$}\ \\

We first calculate the contributions from the four pieces of
$H_{D,k}^{{\rm main}}(p)$. We then combine the results, and compare
to what we would have had if the Fourier coefficients followed the
Sato-Tate distribution or for each prime immediately perfectly
followed the complex multiplication analogue of Sato-Tate.

\begin{lem}\label{lem:fam1S0Fterm} Let $\supp(\hphi) \subset (-\sigma,
\sigma)$. We have \be S_0(\F) \  = \ \phi(0) + \frac{2 \hphi(0)
\cdot(2\gamma_{{\rm PNT}} - \gamma^{(\ge 5)}_{{\rm CM};0} -
\gamma^{(1)}_{2,3})}{\log R} + O\left(\frac1{\log^3 R}\right) +
O(N^{\sigma-1}), \ee where \bea \gamma^{(\ge 5)}_{{\rm CM};0} & \ =
\ & \sum_{p \ge 5}
\frac{4 \log p}{p(p+1)}\ \ \ \ \ \ \  \ \approx \ 0.709919 \nonumber\\
\gamma^{(1)}_{2,3} & \ = \ & \frac{2\log 2}2 + \frac{2\log 3}3 \ \
\approx \ 1.4255554, \eea and $\gamma_{{\rm PNT}}$ is defined in
Lemma \ref{lem:PNTsum2phi}.
\end{lem}

Note $\gamma_{{\rm CM};0}^{(\ge 5)}$ is almost $2\gamma_{{\rm
ST};0}$ (see \eqref{eq:satotatedistrconstants}); the difference is
that here $p\ge 5$.

\begin{proof} Substituting for $A_{0,\F}(p)$ and using
\eqref{eq:taylorexpandhphi} yields \be S_0(\F) \ = \
-\frac{2\hphi(0)}{\log R} \sum_{p \ge 5} \frac{4 \log p}{p(p+1)} + 2
\sum_{p \ge 5} \frac{2 \log p}{p \log R} \hphi\left(2\frac{\log
p}{\log R}\right) + O\left(\frac1{\log^3 R}\right). \ee The first
prime sum converges; using the first million primes we find
$\gamma^{(\ge 5)}_{{\rm CM};0} \approx 0.709919$. The remaining
piece is \be 2 \sum_{p} \frac{2 \log p}{p \log R}\
\hphi\left(2\frac{\log p}{\log R}\right) - \frac{2\hphi(0)}{\log
R}\left(\frac{2\log 2}2 + \frac{2\log 3}3\right) +
O\left(\frac1{\log^3 R}\right). \ee The claim now follows from the
definition of $\gamma^{(1)}_{2,3}$ and using Lemma
\ref{lem:PNTsum2phi} to evaluate the remaining sum.
\end{proof}

\begin{lem} Let $\supp(\hphi) \subset (-\sigma, \sigma)$
and \be \gamma^{(1,3)}_{{\rm CM};2} \ = \ \sum_{p \equiv 1 \bmod 3}
\frac{2(5p^2+2p+1)\log p}{p(p+1)^3} \ \approx \ 0.6412881898.\ee
Then \be S_2(\F) \ = \ -\frac{\phi(0)}2 + \frac{2\hphi(0) \cdot
(-\gamma_{{\rm PNT};1,3} + \gamma^{(1,3)}_{{\rm CM};2})}{\log R} +
O\left(\frac1{\log^3 R}\right) + O(N^{\sigma - 1}), \ee where
$\gamma_{{\rm PNT};1,3}=-2.375494$ (see Lemma
\ref{lem:PNTsum2phi13} for its definition).
\end{lem}

\begin{proof} Substituting our formula for $A_{2,\F}(p)$ and
collecting the pieces yields \be S_2(\F) \ = \ -2\sum_{p\equiv 1
\bmod 3} \frac{2\log p}{\log R} \hphi\left(2\frac{\log p}{\log
R}\right) + \frac{2\hphi(0)}{\log R} \sum_{p \equiv 1 \bmod 3}
\frac{2(5p^2+2p+1)\log p}{p(p+1)^3}. \ee The first sum is evaluated
by Lemma \ref{lem:PNTsum2phi13}. The second sum converges, and was
approximated by taking the first four million primes.
\end{proof}

\begin{lem}\label{lem:CMBkappaterms13} For the families
$\mathcal{F}_{B,\kappa}$: $y^2 = x^3 + B(6T+1)^\kappa$ with $B \in
\{1,2,3,6\}$ and $\kappa \in \{1,2\}$, we have $S_{\tilde A}(\F) =
-2\gamma^{(1,3)}_{{\rm CM}; \tilde A,B,\kappa} \hphi(0)/\log R$ $+$
$O(\log^{-3} R)$, where \bea
\gamma^{(1,3)}_{{\rm CM};\tilde A;1,1} & \ \approx \ & .3437 \nonumber\\
\gamma^{(1,3)}_{{\rm CM};\tilde A;1,2} & \ \approx \ & .4203
\nonumber\\ \gamma^{(1,3)}_{{\rm CM};\tilde A;2,2} & \ \approx \ &
.5670 \nonumber\\ \gamma^{(1,3)}_{{\rm
CM};\tilde A;3,2} & \ \approx \ & .1413 \nonumber\\
\gamma^{(1,3)}_{{\rm CM};\tilde A;6,2} & \ \approx \ & .2620; \eea
the error is at most $.0367$.
\end{lem}

\begin{proof} As the sum converges, we have written a
program in C (using PARI as a library) to approximate the answer. We
used all primes $p \le 48611$ (the first 5000 primes), which gives
us an error of at most about $\frac{8}{\sqrt{p}} \cdot
\frac{p}{p+1-2\sqrt{p}} \approx .0367$. The error should be
significantly less, as this is assuming no oscillation. We also
expect to gain a factor of $1/2$ as half the primes have zero
contribution.
\end{proof}

\begin{rek}
When $\kappa = 1$ a simple change of variables shows that all four
values of $B$ lead to the same behavior. The case of $\kappa = 2$ is
more interesting. If $\kappa = 2$ and $B=1$, then we have the
torsion point $(0,6T+1)$ on the elliptic surface. If $B \in
\{2,3,6\}$ and $\js{B} = 1$ then $(0,6t+1 \bmod p)$ is on the curve
$E_t \bmod p$, while if $\js{B} = -1$ then $(0,6t+1 \bmod p)$ is not
on the reduced curve.
\end{rek}

\subsubsection{Contribution from $H_{D,k}^{{\rm sieve}}(p)$}\ \\

\begin{lem}\label{lem:sievecontrCM13} Notation as in Lemma
\ref{lem:CMBkappaterms13}, the contributions from the $H_{D,k}^{{\rm
sieve}}(p)$ sieved terms to the lower order corrections are \be
-\frac{2(\gamma_{{\rm CM,\ sieve};012}^{(1,3)}\ + \ \gamma_{{\rm
CM,\ sieve};B,\kappa}^{(1,3)})\hphi(0)}{\log R} +
O\left(\frac1{\log^3 R}\right), \ee  \bea \gamma_{{\rm CM,\
sieve};012}^{(1,3)} & \approx & -.004288 \nonumber\\ \gamma_{{\rm
CM,\
sieve};1,1}^{(1,3)} & \ \approx \ &\ \ \ .000446 \nonumber\\
\gamma_{{\rm CM,\
sieve};1,2}^{(1,3)} & \ \approx \ &\ \ \ .000699 \nonumber\\
\gamma_{{\rm CM,\
sieved};2,2}^{(1,3)} & \ \approx \ &\ \ \ .000761 \nonumber\\
\gamma_{{\rm CM,\
sieve};3,2}^{(1,3)} & \ \approx \ &\ \ \ .000125 \nonumber\\
\gamma_{{\rm CM,\ sieve};6,2}^{(1,3)} & \ \approx \ &\ \ \ .000199,
\eea where the errors in the constants are at most $10^{-15}$ (we
are displaying fewer digits than we could!).
\end{lem}

\begin{proof} The presence of the additional factor of $1/p^3$
ensures that we have very rapid convergence. The contribution from
the $r \ge 3$ terms was calculated at the same time as the
contribution in Lemma \ref{lem:CMBkappaterms13}, and is denoted by
$\gamma_{{\rm CM},{\rm sieve};B,\kappa}^{(1,3)}$. The other terms
($r \in \{0,1,2\}$) were computed in analogous manners as before,
and grouped together into $\gamma_{{\rm CM,\ sieve};012}^{(1,3)}$.
\end{proof}

\subsubsection{Results}

We have shown

\begin{thm}\label{thm:ellcurvemainb1thru6kappa} For $\sigma < 2/3$, the $H_{D,k}^{{\rm main}}(p)$ terms
contribute $\phi(0)/2$ to the main term. The lower order correction
from the $H_{D,k}^{{\rm main}}(p)$ and $H_{D,k}^{{\rm sieve}}(p)$
terms is
 \bea & & \frac{2\hphi(0) \cdot(2\gamma_{{\rm PNT}} -
\gamma^{(\ge 5)}_{{\rm CM};0} - \gamma^{(1)}_{2,3} -\gamma_{{\rm
PNT};1,3} + \gamma^{(1,3)}_{{\rm CM};2} - \gamma_{{\rm CM};\tilde
A,B,\kappa}^{(1,3)}-\gamma_{{\rm CM,\
sieve};012}^{(1,3)}-\gamma_{{\rm CM,\
sieve};B,\kappa}^{(1,3)})}{\log R}\nonumber\\ & &\ \ + \
O\left(\frac1{\log^3 R}\right). \eea Using the numerical values of
our constants for the five choices of  $(B,\kappa)$ gives, up to
errors of size $O(\log^{-3} R)$, lower order terms of approximately
\bea B = 1,\kappa = 1: &\ \ & -2.124 \cdot 2\hphi(0)/\log R,
\nonumber\\ B
= 1,\kappa = 2: &\ \ & -2.201 \cdot 2\hphi(0)/\log R, \nonumber\\
B = 2,\kappa = 2: &\ \ & -2.347 \cdot 2\hphi(0)/\log R \nonumber\\
B = 3,\kappa = 2: &\ \ & -1.921 \cdot 2\hphi(0)/\log R \nonumber\\
B = 6,\kappa = 2: &\ \ & -2.042 \cdot 2\hphi(0)/\log R. \eea These
should be contrasted to the family of cuspidal newforms, whose
correction term was \be \gamma_{{\rm PNT}} \cdot
\frac{2\hphi(0)}{\log R} \ \approx \ -1.33258 \cdot
\frac{2\hphi(0)}{\log R}. \ee \end{thm}

\begin{rek}
The most interesting piece in the lower order terms is from the
weighted moment sums with $r \ge 3$ (see Lemma
\ref{lem:CMBkappaterms13}); note the contribution from the sieving
is significantly smaller (see Lemma \ref{lem:sievecontrCM13}). As
each curve in the family has complex multiplication, we expect the
limiting distribution of the Fourier coefficients to differ from
Sato-Tate; however, the coefficients satisfy a related distribution
(it is uniform if we consider the related curve over the quadratic
field; see \cite{Mur}). This distribution is even, and the even
moments are: 2, 6, 20, 70, 252 and so on. In general, the
$2\ell$\textsuperscript{th} moment is $D_\ell = 2
\cdot\foh\ncr{2\ell}{\ell}$ (the factor of $2$ is because the
coefficients vanish for $p\equiv 2 \bmod 3$, so those congruent to
$2$ modulo $3$ contribute double); note the
$2\ell$\textsuperscript{th} moment of the Sato-Tate distribution is
$C_\ell = \frac1{\ell+1} \ncr{2\ell}{\ell}$. The generating function
is \be g_{{\rm CM}}(x) \ = \ \frac{1-\sqrt{1-4x}}{\sqrt{1-4x}} - 2x
\ = \ 6x^2 + 20x^3 + 126x^4 + \cdots \ = \ \sum_{\ell=2}^\infty
D_\ell x^\ell \ee (see sequence A000984 in \cite{Sl}). The contribution from the $r \ge
3$ terms is \be -\frac{2\hphi(0)}{\log R} \sum_{p \equiv 1 \bmod 3}
\frac{(p-1)\log p}{p+1} \sum_{\ell = 2}^\infty D_\ell
\left(\frac{p}{(p+1)^2}\right)^\ell. \ee Using the generating
function, we see that the $\ell$-sum is just $2(3p+1) /
(p-1)(p+1)^2$, so the contribution is \be  -\frac{2\hphi(0)}{\log R}
\sum_{p \equiv 1 \bmod 3} \frac{2(3p+1)\log p}{(p+1)^3} \ = \ -
\frac{2\gamma^{(1,3)}_{{\rm CM};\tilde A}\ \hphi(0)}{\log R}, \ee
where taking the first million primes yields \be
\gamma^{(1,3)}_{{\rm CM};\tilde A}\ \ \approx \ .38184489. \ee It is
interesting to compare the expected contribution from the Complex
Multiplication distribution (for the moments $r \ge 3$) and that
from the Sato-Tate distribution (for the moments $r \ge 3$). The
contribution from the Sato-Tate, in this case, was shown in Lemma
\ref{lem:STCMconstants} to be \be S_{A,0}(\F) \ = \
-\frac{2\gamma_{{\rm ST};\widetilde{\mathcal{A}}}\ \hphi(0)}{\log
R}, \ \ \ \gamma_{{\rm ST}} \ \approx\ 0.4160714430. \ee Note how
close this is to .38184489, the contribution from the Complex
Multiplication distribution.
\end{rek}





\subsection{CM Example: The family $y^2 = x^3 -B(36T+6)(36T+5)x$ over
$\Q(T)$}\label{sec:family00036t65x0}

The analysis of this family proceeds almost identically to the
analysis for the families $y^2 = x^3 + B (6T+1)^\kappa$ over
$\Q(T)$, with trivial modifications because $D(T)$ has two factors;
note no prime can simultaneously divide both factors, and each
factor is of degree $1$. The main difference is that now $a_t(p) =
0$ whenever $p \equiv 3 \bmod 4$ (as is seen by sending $x \to -x$).
We therefore content ourselves with summarizing the main new
feature.

There are two interesting cases. If $B=1$ then the family has rank
$1$ over $\Q(T)$ (see Lemma \ref{lem:rank36t636t5}); note in this
case that we have the point $(36T+6,36T+6)$. If $B=2$ then the
family has rank $0$ over $\Q(T)$. This follows by trivially
modifying the proof in Lemma \ref{lem:rank36t636t5}, resulting in
$\mathcal{A}_{1,\mathcal{F}}(p) = -2p\js{2}$ if $p\equiv 1 \bmod 4$
and $0$ otherwise (which averages to $0$ by Dirichlet's Theorem for
primes in arithmetic progressions).

As with the previous family, the most interesting pieces are the
lower order correction terms from
$S_{\widetilde{\mathcal{A}}}(\mathcal{F})$, namely the pieces from
$H_{D,k}^{{\rm main}}(p)$ and $H_{D,k}^{{\rm sieve}}(p)$ (as we must
sieve). We record the results from numerical calculations using the
first 10,000 primes. We write the main term as $\gamma_{{\rm
CM};\widetilde{\mathcal{A}},B}^{(1,4)}$ (the $(1,4)$ denotes that
there is only a contribution from $p \equiv 1 \bmod 4$) and the
sieve term as $\gamma_{{\rm CM},{\rm sieve};B}^{(1,4)}$. We find
that \be\label{eq:CMsievenosieveCMfamilyadfafafdadadfa}
 \begin{array}{cccccccc} \gamma_{{\rm
CM};\widetilde{\mathcal{A}},1}^{(1,4)} & \ \approx \  & -0.1109 & \
\ & \
\ & \gamma_{{\rm CM},{\rm sieve};1}^{(1,4)} & \ \approx \ & -.0003 \\
\ \\ \gamma_{{\rm CM};\widetilde{\mathcal{A}},2}^{(1,4)} & \ \approx
\   & \ \ 0.6279 &  \ \ & \ \  &
\gamma_{{\rm CM},{\rm sieve};2}^{(1,4)} & \ \approx \ &\ \ \ .0013. \\
         \end{array} \ee

What is fascinating here is that, when $B=1$, the value of
$\gamma_{{\rm CM};\widetilde{\mathcal{A}},B}^{(1,4)}$ is
significantly lower than what we would predict for a family with
complex multiplication. A natural explanation for this is that the
distribution corresponding to Sato-Tate for curves with complex
multiplication cannot be the full story (even in the limit) for a
family with rank. Rosen and Silverman \cite{RoSi} prove

\begin{thm}[Rosen-Silverman]\label{thmsr} Assume Tate's conjecture holds for
a one-parameter family $\mathcal{E}$ of elliptic curves $y^2 = x^3 +
A(T)x + B(T)$ over $\Q(T)$ (Tate's conjecture is known to hold for
rational surfaces). Let $A_{\mathcal{E}}(p) = \frac{1}{p} \sum_{t
\bmod p} a_{t}(p)$.  Then
\begin{eqnarray}
\lim_{X \to \infty} \frac{1}{X} \sum_{p \leq X} -A_\mathcal{E}(p)
\log p\ =\ {\rm rank} \ \mathcal{E}(\Q(T)).
\end{eqnarray}
\end{thm}

Thus if the elliptic curves have positive rank, there is a slight
bias among the $a_t(p)$ to be negative. For a fixed prime $p$ the
bias is roughly of size $-r$ for each $a_t(p)$, where $r$ is the
rank over $\Q(T)$ and each $a_t(p)$ is of size $\sqrt{p}$. While in
the limit as $p\to\infty$ the ratio of the bias to $a_t(p)$ tends to
zero, it is the small primes that contribute most to the lower order
terms. As $\gamma_{{\rm CM};\widetilde{\mathcal{A}},B}^{(1,4)}$
arises from weighted sums of $a_t(p)^3$, we expect this term to be
smaller for curves with rank; this is born out beautifully by our
data (see \eqref{eq:CMsievenosieveCMfamilyadfafafdadadfa}).


\subsection{Non-CM Example: The family $y^2 = x^3 -3x + 12T$ over
$\Q(T)$}\label{sec:family000n3xp12t}

We consider the family $y^2 = x^3 -3x + 12T$ over $\Q(T)$; note this
family does not have complex multiplication. For all $t$ the above
is a global minimal Weierstrass equation, and $a_t(2) = a_t(3) = 0$.
Straightforward calculation (see Appendix \ref{sec:family000-312T}
for details) shows that \bea \twocase{\mathcal{A}_{0,\F}(p)
& \ = \ & }{p-2}{\ \ \ \ \ \ \ \ \ \ \ \ \ \ \ \ \ \ \ \ \ \ \ \ \ if $p \ge 5$}{0}{
\ \ \ \ \ \ \ \ \ \ \ \ \ \ \ \ \ \ \ \ \ \ \ \ \ otherwise}\nonumber\\
\twocase{\mathcal{A}_{1,\F}(p) & \ = \ &}{\js{3} + \js{-3}}{ \ \ \ \ \ \ \ \ \ \ \ \ \ \ \ \
if $p \ge 5$}{0}{ \ \ \ \ \ \ \ \ \ \ \ \ \ \ \ \ otherwise} \nonumber\\
\twocase{\mathcal{A}_{2,\F}(p) & \ = \ & }{p^2 - 2p-2-p\js{-3}}{if
$p \ge 5$}{0}{otherwise.} \eea

Unlike our families with complex multiplication (which only had
additive reduction), here we have multiplicative
reduction\footnote{As we have multiplicative reduction, for each $t$
as $p\to\infty$ the $a_t(p)$ satisfy Sato-Tate; see
\cite{CHT,Tay}.}, and must calculate $\mathcal{A}_{m,\F}'(p)$. We
have \be \threecase{A_{m,\F}'(p) \ = \ }{0}{if $p = 2, 3$}{2}{if $m$
is even}{\js{3}+\js{-3}}{if $m$ is odd;} \ee this follows (see
Appendix \ref{sec:family000-312T}) from the fact that for a given
$p$ there are only two $t$ modulo $p$ such that $p|\Delta(t)$, and
one has $a_t(p) = \js{3}$ and the other has $a_t(p) = \js{-3}$.


We sketch the evaluations of the terms from
\eqref{eq:expansionsievedECfamSF} of Theorem \ref{thm:expSFECs}; for
this family, note that $H_{D,k}(p) = 1$. We constantly use the
results from Appendix \ref{sec:family000-312T}.

\begin{lem} We have $S_{\mathcal{A}'}(\F) =
-2\gamma_{\mathcal{A}'}^{(3)}\hphi(0)/\log R + O(\log^{-3} R)$,
where \be \gamma_{\mathcal{A}'}^{(3)} \ = \ 2\left(\sum_{p\ge
5}\frac{\log p}{p^3-p}+\sum_{p\ge 5 \atop p \equiv 1 \bmod
12}\frac{\log p}{p^2-1} - \sum_{p\ge 5\atop p\equiv 5 \bmod
12}\frac{\log p}{p^2-1}\right) \ \approx \ -0.082971426.\ee
\end{lem}

\begin{proof} As $\mathcal{A}_{m,\F}'(p)=\js{3}^m+\js{-3}^m$, the
result follows by separately evaluating $m$ even and odd, and using
the geometric series formula.
\end{proof}

\begin{lem} We have \be S_{0}(\mathcal{F}) \  =\ \phi(0)
-\frac{2\hphi(0)\cdot
(\gamma_0^{(3)}+\gamma_{2,3}^{(1)}-2\gamma_{{\rm PNT}})}{\log R} +
O\left(\frac1{\log^{3} R}\right),\ee where \be \gamma_0^{(3)} \ = \
\sum_{p\ge 5} \frac{(4p-2)\log p}{p^2(p+1)} \ \approx \ 0.331539448,
\ee $\gamma_{{\rm PNT}}$ is defined in Lemma \ref{lem:PNTsum2phi}
and $\gamma_{2,3}^{(1)}$ is defined in Lemma \ref{lem:fam1S0Fterm}.
\end{lem}

\begin{proof} For $p\ge 5$ we have $\mathcal{A}_{0,\mathcal{F}}(p) =
p-2$. The $\gamma_0^{(3)}$ term comes from collecting the pieces
whose prime sum converges for any bounded $\hphi$ (and replacing
$\hphi(2\log p/\log R)$ with $\hphi(0)$ at a cost of $O(\log^{-2}
R)$), while the remaining pieces come from using Lemma
\ref{lem:PNTsum2phi} to evaluate the prime sum which converges due
to the compact support of $\hphi$.
\end{proof}

\begin{lem} We have $S_{1}(\mathcal{F}) =
-2\gamma_{1}^{(3)}\hphi(0)/\log R + O(\log^{-3} R)$, where \be
\gamma_{1}^{(3)} \ = \ \sum_{p \ge 5} \left[\js{3}+\js{-3}\right]
\cdot \frac{(p-1)\log p}{p^2 (p+1)^2} \ = \ -0.013643784.\ee
\end{lem}

\begin{proof} As the prime sums decay like $1/p^2$, we may replace
$\hphi(\log p/\log R)$ with $\hphi(0)$ at a cost of $O(\log^{-2}
R)$. The claim follows from $\mathcal{A}_{1,\F}(p) = \js{3}+\js{-3}$
and simple algebra.
\end{proof}

\begin{lem} We have \be S_{2}(\mathcal{F}) \ =\ -\frac{\phi(0)}2-\frac{2\hphi(0)\cdot
(\gamma_2^{(3)}-\foh\gamma_{2,3}^{(1)}+\gamma_{{\rm PNT}})}{\log R}
+ O\left(\frac1{\log^{3} R}\right),\ee where \bea \gamma_{2}^{(3)}
&\ = \ & \sum_{p \ge 5} \frac{\left(2- \js{-3}\right)p^4 - (13+
7)\js{-3}p^3 - (25 + 6 \js{-3})p^2 - (16+2\js{-3})p-4)\log p}{p^3
(p+1)^3} \nonumber\\ &\approx & .085627. \eea
\end{lem}

\begin{proof} For $p\ge 5$ we have $\mathcal{A}_{0,\mathcal{F}}(p) =
p^2-2p-2 -\js{-3}p$. The $\gamma_2^{(3)}$ term comes from collecting
the pieces whose prime sum converges for any bounded $\hphi$ (and
replacing $\hphi(2\log p/\log R)$ with $\hphi(0)$ at a cost of
$O(\log^{-2} R)$), while the remaining pieces come from using Lemma
\ref{lem:PNTsum2phi} to evaluate the prime sum which converges due
to the compact support of $\hphi$.
\end{proof}

\begin{lem} We have $S_{\widetilde{\mathcal{A}}}(\mathcal{F}) =
-2\gamma_{\widetilde{\mathcal{A}}}^{(3)}\hphi(0)/\log R +
O(\log^{-3} R)$, where \be \gamma_{\widetilde{\mathcal{A}}}^{(3)} \
\approx \ .3369. \ee
\end{lem}

\begin{proof} As the series converges, this follows by direct
evaluation.
\end{proof}

We have shown

\begin{thm}\label{thm:sdfafadadafafa} The $S_0(\mathcal{F})$ and $S_2(\mathcal{F})$ terms
contribute $\phi(0)/2$ to the main term. The lower order correction
terms are \bea & & -\frac{2\hphi(0)
\cdot\left(\gamma_{\mathcal{A}'}^{(3)} +  \gamma_0^{(3)} +
\gamma_1^{(3)} + \gamma_2^{(3)} +
\gamma_{\widetilde{\mathcal{A}}}^{(3)} + \foh \gamma_{2,3}^{(1)} -
\gamma_{{\rm PNT}}\right)}{\log R} + O\left(\frac1{\log^3 R}\right);
\eea using the calculated and computed values of these constants
gives \be -2.703 \cdot \frac{2\hphi(0)}{\log R} +
O\left(\frac1{\log^3 R}\right).\ee

Our result should be contrasted to the family of cuspidal newforms,
where the correction term was of size \be \gamma_{{\rm PNT}} \cdot
\frac{2\hphi(0)}{\log R} \ \approx \ -1.33258 \cdot
\frac{2\hphi(0)}{\log R}. \ee \end{thm}

\begin{rek}
It is not surprising that our family of elliptic curves has a
different lower order correction than the family of cuspidal
newforms. This is due, in large part, to the fact that we do not
have immediate convergence to the Sato-Tate distribution for the
coefficients. This is exasperated by the fact that most of the
contribution to the lower order corrections comes from the small
primes.
\end{rek}

\appendix


\section{Evaluation of $S_A(\mathcal{F})$ for the family of cuspidal
newforms}\label{sec:evalSAFcuspnewprimelevelN}

\begin{lem}\label{lem:PeterST2} Notation as in \S\ref{sec:cuspnewS}, we have \bea
S_{A}(\F) \ = \  -\frac{2\gamma_{{\rm ST};\widetilde{\mathcal{A}}}\
\hphi(0)}{\log R} +O\left(\frac1{R^{.11}\log^2 R}\right) \ +\
O\left(\frac{\log R}{N^{.73}}\right) +
O\left(\frac{N^{3\sigma/4}\log R}{N}\right).\nonumber\\ \eea In
particular, for test functions supported in $(-4/3,4/3)$ we have
\bea S_{A}(\F) & \ = \ & -\frac{2\gamma_{{\rm
ST};\widetilde{\mathcal{A}}}\ \hphi(0)}{\log R} +
O\left(R^{-\gep}\right), \eea where $\gamma_{{\rm
ST};\widetilde{\mathcal{A}}} \approx .4160714430$ (see Lemma
\ref{lem:STCMconstants}).
\end{lem}

\begin{proof} Recall \bea S_A(\F) & \ = \  & - 2\hphi(0)\sum_{p}
\sum_{r=3}^\infty \frac{A_{r,\F}(p)p^{r/2}(p-1)\log
p}{(p+1)^{r+1}\log R}. \eea Using $|A_{r,\F}(p)| \le 2^r$, we may
easily bound the contribution from $r$ large, say $r \ge 1+2\log R$.
These terms contribute \bea &\ \ll\ & \sum_{p} \sum_{r=1+2\log
R}^\infty\frac{2^r p^{r/2}(p-1)\log p}{(p+1)^{r+1}\log R}
\nonumber\\ & \ll & \frac1{\log R}\sum_{p} \log p \sum_{r=1+2\log
R}^\infty \left(\frac{2\ \sqrt{p}}{p+1}\right)^r \nonumber\\ &\ll &
\frac1{\log R}\sum_{p} \log p \left(\frac{2\
\sqrt{p}}{p+1}\right)^{2\log R} \nonumber\\ & \ll & \frac1{\log
R}\left[ 2007\cdot \left(\frac{2\sqrt{2}}{3}\right)^{2\log R} +
\sum_{p \ge 2008} \frac{\log p}{p^{(2\log R)/3}}\right] \ \ll \
\frac1{R^{.77}\log R}; \eea note it is essential that $2\sqrt{2}/3 <
1$. Thus it suffices to study $r \le 2\log R$.

\bea S_{A}(\F) &\ = \ & - 2\hphi(0)\sum_{p} \sum_{r=3}^{2\log R}
\sum_{k=0}^{r/2} b_{r,r-2k} \frac{A_{r,\F;k}(p)p^{r/2}(p-1)\log
p}{(p+1)^{r+1}\log R}\ + O\left(\frac1{R^{.77}\log R}\right)\nonumber\\
&=& -\frac{2\hphi(0)}{\log R} \sum_p \frac{(p-1)\log p}{p+1}
\sum_{\ell=2}^{\log R} C_\ell\cdot
\left(\frac{p}{(p+1)^2}\right)^\ell +O\left(\frac1{R^{.77}\log
R}\right)\nonumber\\ & & \ \ -\ \frac{2\hphi(0)}{\log
R}\sum_p\sum_{r=3}^{2\log R} \sum_{k=0 \atop k \neq r/2}^{r/2}
b_{r,r-2k} \frac{A_{r,\F;k}(p)p^{r/2}(p-1)\log p}{(p+1)^{r+1}}. \eea

In Lemma \ref{lem:STCMconstants} we handled the first $p$ and
$\ell$-sum when we summed over all $\ell\ge 2$; however, the
contribution from $\ell \ge \log R$ is bounded by $(8/9)^{\log R}
\ll R^{-.11}$. Thus \bea S_{A}(\F) &\ = \ & -\frac{2\gamma_{{\rm
ST};3}\ \hphi(0)}{\log R} +O\left(\frac1{R^{.11}\log R}\right)\nonumber\\
& & \ \ -\ \frac{2\hphi(0)}{\log R}\sum_p\sum_{r=3}^{2\log R}
\sum_{k=0}^{(r-2)/2} b_{r,r-2k} \frac{A_{r,\F;k}(p)p^{r/2}(p-1)\log
p}{(p+1)^{r+1}}.\ \eea

To finish the analysis we must study the $b_{r,r-2k} A_{r,\F;k}(p)$
terms. Trivial estimation suffices for all $r$ when $p \ge 13$; in
fact, bounding these terms for small primes is what necessitated our
restricting to $r \le 2\log R$. From \eqref{eq:PeterssonFormula}
(the Petersson formula with harmonic weights) we find \bea
A_{r,\F;k}(p) \ \ll \ \frac{p^{(r-2k)/4} \log\left(p^{(r-2k)/4}
N\right)}{k^{5/6}N} \ \ll \ \frac{r p^{r/4} \log(pN)}{N}. \eea As
$|\sum_{k=0}^{(r-2)/2} b_{r,r-2k}| \le 2^r$, we have \bea S_{A}(\F)
\ = \ -\frac{2\gamma_{{\rm ST};\widetilde{\mathcal{A}}}\
\hphi(0)}{\log R} +O\left(\frac1{R^{.11}\log R}\right) \ +\
O\left(\frac{1}{N}\sum_p\sum_{r=3}^{2\log R} \frac{r 2^r
p^{3r/4}\log (pN)}{(p+1)^r\log R}\right). \eea As our Schwartz test
functions restrict $p$ to be at most $R^\sigma$, the second error
term is bounded by \bea &\ \ll\ & \frac1{N\log R} \sum_p \log(pN)
\sum_{r=3}^{2\log R} r \left(\frac{2p^{3/4}}{p+1}\right)^r
\nonumber\\ & \ll & \frac{\log R}N\left[\sum_{p \le 2007}
\sum_{r=3}^{2\log R} \left(\frac{2p^{3/4}}{p+1}\right)^r + \sum_{p
\ge 2008}\sum_{r=3}^{2\log R}  \left(\frac{2p^{3/4}}{p+1}\right)^r
\right] \nonumber\\ &\ll & \frac{\log R}N\left[2007
\left(\frac{2\cdot 3^{3/4}}{4}\right)^{2\log R}\log R + \sum_{p \ge
2008} \frac{2p^{3/4}}{p+1} \right]\nonumber\\ &\ll &
\frac{N^{.27}\log^2 R}{N} + \frac{\log R}{N} \sum_{p =
2011}^{R^\sigma} p^{-1/4}\ \ll \ \frac{\log^2 R}{N^{.73}} +
\frac{N^{3\sigma/4}\log R}{N}, \eea which is negligible provided
that $\sigma < 4/3$.
\end{proof}


\section{Evaluation of $A_{r,\F}$ for families of elliptic curves}
\setcounter{equation}{0}

The following standard result allows us to evaluate the second
moment of many one-parameter families of elliptic curves over $\Q$ (see \cite{ALM,BEW} for a proof).

\begin{lem}[Quadratic Legendre Sums]\label{labquadlegsum} Assume
$a$ and $b$ are not both zero mod $p$ and $p > 2$. Then
\begin{equation}
\twocase{\zsum{t} \js{at^2 + bt + c}\ = \ }{(p-1)\js{a}}{if $p\
\notdiv b^2 - 4ac$}{-\js{a}}{otherwise.}
\end{equation}
\end{lem}

\subsection{The family $y^2 = x^3 + B (6T+1)^\kappa$ over
$\Q(T)$}\label{sec:calc9tp1fam}\ \\

In the arguments below, we constantly use the fact that if
$p|\Delta(t)$ then $a_t(p) = 0$. This allows us to ignore the $p \
\notdiv \Delta(t)$ conditions. We assume $B \in \{1,2,3,6\}$ and
$\kappa \in \{1,2\}$.

\begin{lem} We have \be \twocase{\mathcal{A}_{0,\F}(p)
 \ = \  }{p-1}{if $p \ge
5$}{0}{otherwise.} \ee \end{lem}

\begin{proof} We have $\mathcal{A}_{0,\F}(p) = 0$ if $p = 2$ or $3$ because,
in these cases, there are no $t$ such that $p \ \notdiv \Delta(t)$.
If $p \ge 5$ then $p \ \notdiv \Delta(t)$ is equivalent to $p\
\notdiv B(6t+1) \bmod p$. As $6$ is invertible mod $p$, as $t$
ranges over $\Z/p\Z$ there is exactly one value such that $B(6t+1)
\equiv 0 \bmod p$, and the claim follows. \end{proof}

\begin{lem} We have $\mathcal{A}_{1,\F}(p) = 0$. \end{lem}

\begin{proof} The claim is immediate for $p = 2, 3$ or $p \equiv 2
\bmod 3$; it is also clear when $\kappa = 1$. Thus we assume below
that $p \equiv 1 \bmod 3$ and $\kappa = 2$:
\begin{eqnarray}
-\mathcal{A}_{1,\mathcal{F}}(p) &\ =  \ & \sum_{t \bmod p} a_t(p) \nonumber\\
&=& \sum_{t \bmod p} \sum_{x \bmod p} \js{x^3 + B(6t+1)^2} \ = \
 \sum_{t \bmod p} \sum_{x \bmod p} \js{x^3+Bt^2}. 
\end{eqnarray} The $x=0$ term gives $\js{B}(p-1)$, and the remaining $p-1$ values of $x$
each give $-\js{B}$ by Lemma\ref{labquadlegsum}. Therefore
$\mathcal{A}_{1,\mathcal{F}}(p) = 0$.
\end{proof}

\begin{lem} We have $\mathcal{A}_{2,\F}(p)
= 2p^2 - 2p$ if $p \equiv 1 \bmod 3$, and 0 otherwise. \end{lem}

\begin{proof} The claim is immediate for $p = 2, 3$ or $p \equiv 2
\bmod 3$. We do the proof for the harder case of $\kappa = 2$; the
result is the same when $\kappa = 1$ and follows similarly. For $p
\equiv 1 \bmod 3$:
\begin{eqnarray}
\mathcal{A}_{2,\mathcal{F}}(p) \ =\  \sum_{t \bmod p} a_t^2(p) &\ =
\ & \sum_{t \bmod p} \sum_{x \bmod p}\sum_{y \bmod p}
\js{x^3+B(6t+1)^2} \js{y^3+B(6t+1)^2} \nonumber\\
&=& \sum_{t \bmod p} \sum_{x \bmod p}\sum_{y \bmod p}
\js{x^3+Bt^2}\js{y^3+Bt^2} \nonumber\\ &=& \sum_{t=1}^{p-1} \sum_{x
(p)}\sum_{y \bmod p} \js{x^3+Bt^2}\js{y^3+Bt^2} \nonumber\\ &=&
\sum_{t=1}^{p-1} \sum_{x \bmod p}\sum_{y \bmod p} \js{t^4}
\js{tx^3+B}\js{ty^3+B} \nonumber\\ &=& \sum_{x \bmod p}\sum_{y \bmod
p} \sum_{t \bmod p} \js{tx^3+B}\js{ty^3+B} - p^2\js{B^2}.
\end{eqnarray}

We use inclusion / exclusion to reduce to $xy \neq 0$. If $x=0$, the
$t$ and $y$-sums give $p\js{B}\js{B}$. If $y=0$, the $t$ and
$x$-sums give $p\js{B}\js{B}$. We subtract the doubly counted
contribution from $x=y=0$, which gives $p\js{B}\js{B}$. Thus
\begin{eqnarray}
A_{2,\mathcal{F}}(p) &\ =  \ & \sum_{x=1}^{p-1}\sum_{y=1}^{p-1}
\sum_{t \bmod p} \js{tx^3+B}\js{ty^3+B} + 2p -p - p^2.
\end{eqnarray}

By Lemma \ref{labquadlegsum}, the $t$-sum is $(p-1)\js{x^3y^3}$ if
$p|B^2(x^3-y^3)^2$ and $-\js{x^3y^3}$ otherwise; as $B|6^\infty$ we
have $p\ \notdiv B$. As $p = 6m+1$, let $g$ be a generator of the
multiplicative group $\Z/p\Z$. Solving $g^{3a} \equiv g^{3b}$ yields
$b = a$, $a + 2m$, or $a + 4m$, so $x^3 \equiv y^3$ three times (for
$x, y \not\equiv 0 \bmod p$). In each instance $y$ equals $x$ times
a square ($1$, $g^{2m}$, $g^{4m}$). Thus
\begin{eqnarray}
\mathcal{A}_{2,\mathcal{F}}(p) &\ =\ & \sum_{x=1}^{p-1} \sum_{y=1
\atop y^3 \equiv x^3}^{p-1} p - \sum_{x=1}^{p-1}\sum_{y=1}^{p-1}
\js{x^3y^3} + p - p^2 \nonumber\\ &=& (p-1)3p + p - p^2 \ = \  2p^2
- 2p.
\end{eqnarray}
\end{proof}

\subsection{The family $y^2 = x^3 -(36T+6)(36T+5)x$ over
$\Q(T)$}\label{sec:family36tp636tp5}

In the arguments below, we constantly use the fact that if
$p|\Delta(t)$ then $a_t(p) = 0$. This allows us to ignore the $p \
\notdiv \Delta(t)$ conditions.

\begin{lem}\label{lem:rank36t636t5} We have $\mathcal{A}_{0,\F}(p)
 = p-2$ if $p \ge 3$ and 0 otherwise. \end{lem}

\begin{proof} We have $\mathcal{A}_{0,\F}(p) = 0$ if $p = 2$ because
there are no $t$ such that $p \ \notdiv \Delta(t)$. If $p \ge 3$
then $p \ \notdiv \Delta(t)$ is equivalent to $p\ \notdiv
(36t+6)(36t+5) \bmod p$. As $36$ is invertible mod $p$, as $t$
ranges over $\Z/p\Z$ there are exactly two values such that
$(36t+6)(36+5) \equiv 0 \bmod p$, and the claim follows.
\end{proof}

\begin{lem} We have $\mathcal{A}_{1,\F}(p) = -2p$ if $p\equiv 1 \bmod 4$
and 0 otherwise. \end{lem}

\begin{proof} The claim is immediate if $p=2$ or $p \equiv 3 \bmod 4$. If $p \equiv 1 \bmod 4$
then we may replace $36t+6$ with $t$ in the complete sums, and we
find that
\begin{eqnarray}
\mathcal{A}_{1,\mathcal{F}}(p)  =  -\sum_{t \bmod p} \sum_{x \bmod
p} \js{x^3-t(t-1)x} =  -\sum_{x \bmod p} \js{-x} \sum_{t \bmod p}
\js{t^2 - t -x^2}.
\end{eqnarray}
As $p\equiv 1 \bmod 4$, $-1$ is a square, say $-1 \equiv \alpha^2
\bmod p$. Thus $\js{-x} = \js{x}$ above. Further by Lemma
\ref{labquadlegsum} the $t$-sum is $p-1$ if $p$ divides the
discriminant $1+4x^2$, and is $-1$ otherwise. There are always
exactly two distinct solutions to $1+4x^2 \equiv 0 \bmod p$ for $p
\equiv 1 \bmod 4$, and both roots are squares modulo $p$.

To see this, letting $\overline{w}$ denote the inverse of $w$ modulo
$p$ we find the two solutions are $\pm \overline{2}\alpha$. As
$\js{\overline{w}} = \js{w}$ and $\js{-1} = 1$, we have
$\js{\overline{2}\alpha} = \js{2\alpha}$. Let $p = 4n+1$. Then
$\js{2} = (-1)^{(p^2-1)/8} = (-1)^n$, and by Euler's criterion we
have \be \js{\alpha}\ \equiv \ \alpha^{(p-1)/2} \ \equiv \
\left(\alpha^2\right)^{(p-1)/4} \ \equiv \ (-1)^n \bmod p. \ee Thus
$\js{2\alpha} = 1$, and the two roots to $1+4x^2 \equiv 0 \bmod p$
are both squares. Therefore \bea \mathcal{A}_{1,\mathcal{F}}(p) \ =
\  -2p + \sum_{x \bmod p} \js{x} \ = \ -2p. \eea
\end{proof}

\begin{rek} By the results of Rosen and Silverman \cite{RoSi},
our family has rank $1$ over $\Q(T)$; this is not surprising as we
have forced the point $(36T+6, 36T+6)$ to lie on the curve over
$\Q(T)$. \end{rek}

\begin{lem} Let $E$ denote the elliptic curve $y^2 = x^3 - x$,
with $a_E(p)$ the corresponding Fourier coefficient. We have \be
\twocase{\mathcal{A}_{2,\F}(p)
 \ = \  }{2p(p-3) - a_E(p)^2}{if $p \equiv
1 \bmod 4$}{0}{otherwise.}\ee \end{lem}

\begin{proof} The proof follows by similar calculations as above.
\end{proof}

\subsection{The family $y^2 = x^3 -3x+12T$ over
$\Q(T)$}\label{sec:family000-312T}

For the family $y^2 = x^3 - 3x + 12T$, we have \bea c_4(T) & \ = \ &
2^4 \cdot 3^2 \nonumber\\ c_6(T) &=& 2^7 \cdot 3^4 T \nonumber\\
\Delta(T) &=& 2^6 \cdot 3^3 (6T-1)(6T+1); \eea further direct
calculation shows that $a_t(2) = a_t(3) = 0$ for all $t$. Thus our
equation is a global minimal Weierstrass equation, and we need only
worry about primes $p \ge 5$. Note that $c_4(t)$ and $\Delta(t)$ are
never divisible by a prime $p\ge 5$; thus this family can only have
multiplicative reduction for primes exceeding $3$.

If $p|6t-1$, replacing $x$ with $x+1$ (to move the singular point to
$(0,0)$) gives $y^2 - 3x^2 \equiv x^3 \bmod p$. The reduction is
split if $\sqrt{3} \in \mathbb{F}_p$ and non-split otherwise. Thus
if $p|6t-1$ then $a_t(p) = \js{3}$. A similar argument (sending $x$
to $x-1$) shows that if $p|6t+1$ then $a_t(p) = \js{-3}$. A
straightforward calculation shows \bea\twocase{\js{3}  \ = \  }{\ \
\ 1}{if $p\equiv 1, 11 \bmod 12$}{-1}{if $p \equiv 5,\ \ 7 \bmod
12$,}\ \ \ \    \twocase{\js{-3}  \ = \  }{\ \ \ 1}{if $p\equiv 1,\
\ 7 \bmod 12$}{-1}{if $p \equiv 5, 11 \bmod 12$.} \eea

\begin{lem} We have $\mathcal{A}_{0,\F}(p)
= p-2$ if $p \ge 3$ and 0 otherwise.  \end{lem}

\begin{proof} We have $\mathcal{A}_{0,\F}(p) = 0$ if $p = 2$ or $3$ by
direct computation. As $12$ is invertible mod $p$, as $t$ ranges
over $\Z/p\Z$ there are exactly two values such that $(6t-1)(6t+1)
\equiv 0 \bmod p$, and the claim follows.
\end{proof}

\begin{lem} $\mathcal{A}_{1,\F}(2) = \mathcal{A}_{1,\F}(3) = 0$,
and for $p \ge 5$ we have \bea \threecase{\mathcal{A}_{1,\F}(p) \ =
\ \js{3} + \js{-3} \ = \ } {\ \ \ 2}{if $p \equiv 1 \bmod 12$}{\ \ \
0}{if $p\equiv 7, 11 \bmod 12$}{-2}{if $p \equiv 5 \bmod 12$.}\eea
\end{lem}

\begin{proof} The claim is immediate for $p \le 3$. We have \bea
\mathcal{A}_{1,\F}(p) &=& -\sum_{t \bmod p \atop \Delta(t)
\not\equiv 0 \bmod p} a_t(p) \nonumber\\ &=& -\sum_{t\bmod p}
\js{x^3 - 3x + 12t} + \sum_{t \bmod p \atop \Delta(t) \equiv 0 \bmod
p} \js{x^3-3x+12} \nonumber\\ &=& 0 + \js{3} + \js{-3};\eea the last
line follows from our formulas for $a_t(p)$ for $p|\Delta(t)$.
\end{proof}

\begin{lem} $\mathcal{A}_{2,\F}(2) = \mathcal{A}_{2,\F}(3) = 0$,
and for $p \ge 5$ we have $\mathcal{A}_{2,\F}(p) = p^2 -3p - 4 -
2\js{-3}$.
 \end{lem}

\begin{proof} The claim is immediate for $p \le 3$. For $p\ge 5$ we
have $a_t(p)^2 = 1$ if $p|\Delta(t)$. Thus \bea
\mathcal{A}_{2,\F}(p) & \ = \ & \sum_{t \mod p \atop \Delta(t)
\not\equiv 0 \bmod p} a_t(p)^2 \nonumber\\ &=& \sum_{t\mod p}
\sum_{x\bmod p}\sum_{y\bmod p} \js{x^3-3x+12t}\js{y^3-3y+12t} - 2.
\eea Sending $t\to 12^{-1} t \bmod p$, we have a quadratic in $t$
with discriminant \be \left((x^3-3x) - (y^3-3y)\right)^2 \ = \
(x-y)^2 \cdot (y^2 + xy + x^2 - 3)^2 \ = \ \delta(x,y). \ee

We use Lemma \ref{labquadlegsum} to evaluate the $t$-sum; it is
$p-1$ if $p|\delta(x,y)$, and $-1$ otherwise. Letting $\eta(x,y) =
1$ if $p|\delta(x,y)$ and $0$ otherwise, we have \bea
\mathcal{A}_{2,\F}(p) & \ = \ & \sum_{x \bmod p}\sum_{y\bmod p}
\eta(x,y) p - p^2 - 2. \eea

For a fixed $x$, $p|\delta(x,y)$ if $y=x$ or if $y^2 + xy + x^2 -
3\equiv 0 \bmod p$ (we must be careful about double counting). There
are two distinct solutions to the quadratic (in $y$) if its
discriminant $12-3x^2$ is a non-zero square in $\Z/p\Z$, one
solution (namely $-2^{-1}x$, which is not equivalent to $x$) if it
is congruent to zero (which happens only when $x \equiv \pm 2 \bmod
p$), and no solutions otherwise. If the discriminant $12-3x^2$ is a
square, the two solutions are distinct from $x$ provided that $x
\not\equiv \pm 1 \bmod p$ (if $x\equiv \pm 1 \bmod p$ then one of
the solutions is $x$ and the other is distinct). Thus, for a fixed
$x$, the number of $y$ such that $p|\delta(x,y)$ is $2+
\js{12-3x^2}$ if $x\not\equiv \pm 1, \pm 2$ and $2$ if $x\equiv \pm
1, \pm2$. Therefore \bea \mathcal{A}_{2,\F}(p) & \ = \ & \sum_{x
\bmod p \atop x\not\equiv \pm 1, \pm 2 \bmod p} \left[2+
\js{12-3x^2}\right]\cdot p + \sum_{x\equiv \pm 1, \pm 2 \bmod p} 2
\cdot p - p^2 - 2 \nonumber\\ & = & 2(p-4)p + p\sum_{x \bmod p \atop
x\not\equiv \pm 1, \pm 2 \bmod p} \js{12-3x^2} + 4\cdot 2p-p^2-2
\nonumber\\ &=& p^2 -2 +p \sum_{t\bmod p}\js{12-3x^2} - 2p \ = \ p^2
-2p-2 - p\js{-3}, \eea where we used Lemma \ref{labquadlegsum} to
evaluate the $x$-sum (as $p \ge 5$, $p$ does not divide its
discriminant).
\end{proof}



\bigskip

\end{document}